\documentclass[a4paper,11pt,nonatbib]{plan-base}
\usepackage[numbers]{natbib}
\usepackage{framed}
\usepackage{xcolor}
\usepackage{tikz}

\usepackage{soul}

\usepackage{color}

\usepackage[english]{babel}
\usepackage[labelfont=bf]{caption}
\usepackage[labelfont=bf]{subcaption}
\usepackage{comment}
\usepackage{footmisc}
\usepackage{bbm}
\usepackage{amsmath}
\usepackage{amssymb}
\usepackage{amsthm}
\usepackage{ifpdf}

\ifpdf
    \usepackage[pdftex]{hyperref}
\else
    \usepackage[hypertex]{hyperref}
\fi

\usepackage{color}
\definecolor{hrefcolor}{rgb}{0.0,0.5,0.8}
\definecolor{hlgreen}{rgb}{0,0.7,0}

\hypersetup{
   colorlinks=true,
   linkcolor=hrefcolor,
   citecolor=hlgreen,
   filecolor=hrefcolor,
   urlcolor=hrefcolor,
}

\ifdefined\tikz
\else
    \ifpdf
        \usepackage[pdftex]{graphicx}
    \else
        \usepackage{graphicx}
        
    \fi

\fi

\newenvironment{enumroman}
    {
     
     \begin{enumerate}
        \setlength{\leftmargin}{3.0em}
        \setlength{\labelwidth}{2.5em}
        \setlength{\labelsep}{0.5em}
    }
    {\end{enumerate}}

\newcounter{remcount}

\newtheorem{theorem}{Theorem}
\newtheorem{corollary}{Corollary}
\newtheorem{lemma}{Lemma}
\newtheorem{proposition}{Proposition}

\theoremstyle{definition}
\newtheorem{definition}{Definition}

\newtheorem*{assumption*}{Assumption}
\newtheorem{remark}{Remark}
\newtheorem*{remark*}{Remark}
\newtheorem*{definition*}{Definition}
\newtheorem{example}{Example}

\numberwithin{equation}{section}
\numberwithin{lemma}{section}
\numberwithin{theorem}{section}
\numberwithin{proposition}{section}
\numberwithin{definition}{section}
\numberwithin{remark}{section}
\numberwithin{example}{section}
\numberwithin{assumption}{section}
\numberwithin{algorithm}{section}
\numberwithin{corollary}{section}

\newcommand*{\doi}[1]{doi:\href{http://dx.doi.org/#1}{\detokenize{#1}}}

\makeatletter
\AtBeginDocument{
    \hypersetup{
        pdftitle = {\@title},
        pdfauthor = {\@author}
    }
}
\makeatother

\newcommand{\field}[1]{\mathbb{#1}}

\newcommand{\R}{\field{R}}

\newcommand{\C}{\field{C}}
\newcommand{\norm}[1]{\|#1\|}

\newcommand{\abs}[1]{|#1|}

\newcommand{\inv}[1]{#1^{-1}}
\newcommand{\grad}[1]{\nabla #1}
\newcommand{\freevar}{\,\boldsymbol\cdot\,}

\newcommand{\Union}\bigcup
\newcommand{\Isect}\bigcap
\newcommand{\union}\cup
\newcommand{\isect}\cap
\newcommand{\bigunion}\bigcup
\newcommand{\bigisect}\bigcap

\newcommand{\defeq}{:=}

\newcommand{\downto}{\searrow}
\newcommand{\upto}{\nearrow}
\newcommand{\subdiff}{\partial}
\newcommand{\rangeSymbol}{\mathcal{R}}
\newcommand{\range}[1]{\rangeSymbol(#1)}

\DeclareMathOperator*{\argmin}{arg\,min}

\DeclareMathOperator{\interior}{int}

\DeclareMathOperator{\Dom}{dom}

\makeatletter
\def \uminus@sym{\setbox0=\hbox{$\cup$}\rlap{\hbox 
        to\wd0{\hss\raise0.5ex\hbox{$\scriptscriptstyle{-}$}\hss}}\box0}
    \def \uminus    {\mathrel{\uminus@sym}}
\makeatother

\newcommand{\mathvar}[1]{\textup{#1}}

\renewcommand{\tilde}{\widetilde}

\newcommand{\iprod}[2]{\langle #1,#2\rangle}

\newcommand{\BDspace}{\mathvar{BD}}

\newcommand{\BVspace}{\mathvar{BV}}

\newcommand{\Meas}{\mathcal{M}}

\renewcommand{\L}{\mathcal{L}}

\newcommand{\BD}{\partial}
\renewcommand{\d}{\,d} 

\newcommand{\TGV}{\mathvar{TGV}}
\newcommand{\TV}{\mathvar{TV}}

\DeclareMathOperator{\divergence}{div}

\newcommand{\weakto}{\mathrel{\rightharpoonup}}
\makeatletter
\def \weaktostar@sym{\setbox0=\hbox{$\rightharpoonup$}\rlap{\hbox 
        to\wd0{\hss\raise1ex\hbox{$\scriptscriptstyle{*\,}$}\hss}}\box0}
    \def \weaktostar    {\mathrel{\weaktostar@sym}}
\makeatother

\theoremstyle{definition}
\newtheorem{assumptionx}{Assumption}
\newenvironment{namedassumption}[2][]
    {\renewcommand\theassumptionx{A-#2}\begin{assumptionx}[#1]}
    {\end{assumptionx}}

\DeclareMathOperator{\Sgn}{Sgn}
\DeclareMathOperator{\KRTV}{KRTV}

\newcommand{\nullspace}[1]{\mathcal{N}(#1)}

\newcommand{\setto}{\rightrightarrows}

\def\DIMurange{m}
\def\DIMdomain{n}
\def\DIMkurange{d}
\def\pixels{\ell}

\def\alphavec{\vec\alpha}

\def\ee{\text{\textsc{e}}}
\def\costltwo{{\ensuremath{L_2^2}}}
\def\costlone{{\ensuremath{L_\eta^1}}}
\def\costhubertv{{\ensuremath{L_\eta^1\!\nabla}}}
\def\costsymbreg{{\ensuremath{B_sL_1\!\nabla_\eta}}}
\def\NA{N}

\def\Smoother{H}
\newcommand{\JorigSYM}{J}
\newcommand{\JhuberSYM}{J^{\gamma,0}}
\newcommand{\JepsilonSYM}{J^{\gamma,\epsilon}}

\newcommand{\JepsilonXXSYM}[1]{J^{#1}}

\newcommand{\Jorig}[2][\lambda, \alpha]{\JorigSYM(#2; #1)}
\newcommand{\Jhuber}[2][\lambda, \alpha]{\JhuberSYM(#2; #1)}
\newcommand{\Jepsilon}[2][\lambda, \alpha]{\JepsilonSYM(#2; #1)}

\newcommand{\JepsilonXX}[3][\lambda, \alpha]{\JepsilonXXSYM{#3}(#2; #1)}

\def\SolM{\mathcal{S}}
\def\ValM#1{\mathcal{I}_{#1}}
\def\ValMap#1{\ValM{\gamma,#1}}

\def\costf{F}
\newcommand{\dprod}[2]{(#1 | #2)}
\newcommand{\huber}[2][\gamma]{\abs{#2}_{#1}}
\newcommand{\huberradon}[2][\gamma,\Meas]{\norm{#2}_{#1}}
\newcommand{\radon}[2][\Meas]{\norm{#2}_{#1}}
\newcommand{\SPACEj}[1][j]{\Meas(\Omega; \R^{m_{#1}})}
\newcommand{\radonj}[2][j]{\radon[{#1}]{#2}}
\newcommand{\huberjX}[3][j]{\huberradon[{#3,#1}]{#2}}
\newcommand{\huberj}[2][j]{\huberjX[#1]{#2}{\gamma}}
\def\borel{\mathcal{B}}

\def\doublebar#1{\bar{\bar{#1}}}

\def\SPACEalpha{\mathcal{P}_\alpha}
\def\SPACEalphaPos{\SPACEalpha}
\def\SPACEalphaCompactPos{\SPACEalpha^\infty}
\def\SPACEalphaCompactPosInt{\interior \SPACEalpha^\infty}

\def\SPACEalphaPosInt{\interior \SPACEalphaPos}

\def\SPACEalphaC{C_0(\Omega; \R^\NA)}

\def\SPACEalphaRC{[0, \infty]^\NA}

\def\SPACEalphaRCInt{(0, \infty]^\NA}

\def\costK{}
\def\invK{K^\dagger}
\def\barK{(K^\dagger)^*}

\def\SPACEuHilbert{H^1(\Omega; \R^\DIMurange)}

\def\SPACEuHilbertCons{H^1(\Omega; \R^\DIMurange)}

\def\SPACEgraduLtwo{L^2(\Omega; \R^{\DIMurange \times \DIMdomain})}

\def\SPACEgradkuLtwo{L^2(\Omega; \R^{\DIMkurange \times \DIMdomain})}
\newcommand{\SPACEkuLp}[2][\DIMkurange]{L^{#2}(\Omega; \R^{#1})}
\newcommand{\SPACEkuLone}[1][\DIMkurange]{\SPACEkuLp[#1]{1}}
\newcommand{\SPACEkuLtwo}[1][\DIMkurange]{\SPACEkuLp[#1]{2}}

\newcommand{\ICTV}{\mathvar{ICTV}}
\newcommand{\POLYTV}{\mathvar{POLYTV}}

\newcommand{\regf}[2][\alpha]{\regfX[#1]{#2}{\gamma}}
\newcommand{\regfX}[3][\alpha]{R^{#3}(#2; #1)}
\newcommand{\regforigSYM}{R}
\newcommand{\regforig}[2][\alpha]{\regforigSYM(#2; #1)}
\newcommand{\regfMARGhuberX}[3][\bar\alpha]{T^{#3}(#2; #1)}
\newcommand{\regfMARGhuber}[2][\bar\alpha]{\regfMARGhuberX[#1]{#2}{\gamma}}
\newcommand{\regfMARG}[2][\bar\alpha]{T(#2; #1)}

\newlength{\colw}

\def\iftgv#1#2#3{{\def\iftgvpar{#1}\def\iftgvtgv{tgv2}\ifx\iftgvpar\iftgvtgv#2\else#3\fi}}
\def\ifkrtv#1#2#3{{\def\ifkrtvpar{#1}\def\ifkrtvkrtv{krtv2}\ifx\ifkrtvpar\ifkrtvkrtv#2\else#3\fi}}
\def\ifictv#1#2#3{{\def\ifictvpar{#1}\def\ifictvictv{ictv}\ifx\ifictvpar\ifictvictv#2\else#3\fi}}
\def\ifpolytv#1#2#3{{\def\ifpolytvpar{#1}\def\ifpolytvpolytv{polytv}\ifx\ifpolytvpar\ifpolytvpolytv#2\else#3\fi}}
\def\iftwoparam#1#2#3{\iftgv{#1}{#2}{\ifictv{#1}{#2}{#3}}}
\def\ifhuber#1#2#3{{\def\ifhpar{#1}\def\ifhh{huberonly}\ifx\ifhpar\ifhh#2\else#3\fi}}
\def\ifsymbreg#1#2#3{{\def\ifhpar{#1}\def\ifhh{symbregman}\ifx\ifhpar\ifhh#2\else#3\fi}}
\def\ifssn#1#2#3{{\def\ifhpar{#1}\def\ifhh{ssn}\ifx\ifhpar\ifhh#2\else#3\fi}}
\def\mathname#1{\iftgv{#1}{$\TGV^2$}{\ifkrtv{#1}{$\KRTV$}{\ifictv{#1}{$\ICTV$}{\ifpolytv{#1}{$\POLYTV_1$}{$\TV$}}}}}
\def\costname#1{\ifhuber{#1}{\costhubertv}{\ifsymbreg{#1}{\costsymbreg}{\costltwo}}}

\def\XXint#1#2#3{{\setbox0=\hbox{$#1{#2#3}{\int}$ }
\vcenter{\hbox{$#2#3$ }}\kern-.6\wd0}}

\newcommand{\resplotxx}[2][]{
    \setlength{\colw}{\textwidth}
    \begin{tikzpicture}
        \pgftext[at=\pgfpoint{0}{0},left,bottom]{%
            \includegraphics[width=\colw]{{#2}.png}
        }
        #1
    \end{tikzpicture}%
}

\newlength{\imw}
\newcommand{\inplot}[3][]{
    \begin{subfigure}[t]{\imw}%
    \resplotxx[#1]{img/#2}
    \caption{#3}
    \ifdefined\subfigprefix\label{\subfigprefix:#2}\else\relax\fi
    \end{subfigure}
    }

\newlength{\scf}

\newcommand{\restabll}[6][]{%
    \input{resimg/#2-vals.tex}%
    \mathname{#3} &  
    \costname{#4} & 
    #5 & 
    \ifpolytv{#3}{%
        $(\RESalphaSCone,\RESalphaSCtwo,\RESalphaSCthree,\RESalphaSCfour)/\pixels$%
    }{%
        \iftwoparam{#3}%
              {$(\RESalphaSCtwo/\pixels^2, \RESalphaSCone/\pixels)$}%
              {$\RESalphaSCone/\pixels$}
    } &
   \RESdist & 
   \RESssim & 
   \RESpsnr &
   \RESiters &
   \def\arg{#6}
   \def\nofig{}
   \ifx\arg\nofig\else\ref{#6}(\subref{#6:#3-#4})\fi \\%
}

\begin{document}

\title{The structure of optimal parameters for image restoration problems}
\author{Juan Carlos De los Reyes, Carola-Bibiane Sch\"onlieb and Tuomo Valkonen}

\maketitle

\begin{abstract}
We study the qualitative properties of optimal regularisation parameters in variational models for image restoration. The parameters are solutions of bilevel optimisation problems with the image restoration problem as constraint. A general type of regulariser is considered, which encompasses total variation (TV), total generalized variation (TGV) and infimal-convolution total variation (ICTV). We prove that under certain conditions on the given data optimal parameters derived by bilevel optimisation problems exist. A crucial point in the existence proof turns out to be the boundedness of the optimal parameters away from $0$ which we prove in this paper. The analysis is done on the original -- in image restoration typically non-smooth variational problem -- as well as on a smoothed approximation set in Hilbert space which is the one considered in numerical computations. For the smoothed bilevel problem we also prove that it $\Gamma$ converges to the original problem as the smoothing vanishes. All analysis is done in function spaces rather than on the discretised learning problem.
\end{abstract}

\section{Introduction}

In this paper we consider the general variational image reconstruction problem that, given parameters $\alpha=(\alpha_1,\ldots,\alpha_N)$, $N\geq 1$, aims to compute an image
$$
    u_{\alpha} \in \argmin_{u \in X} \Jorig[\alpha]{u}.
$$
The image depends on $\alpha$ and belongs in our setting to a generic function space $X$. Here $J$ is a generic energy modelling our prior knowledge on the image $u_\alpha$. The quality of the solution $u_\alpha$ of variational imaging approaches like this one crucially relies on a good choice of the parameters $\alpha$. We are particularly interested in the case
$$
\Jorig[\alpha]{u} = \Phi(Ku)
        +
        \sum_{j=1}^{\NA} \alpha_j \radonj{A_j u},
$$
with $K$ a generic bounded forward operator, $\Phi$ a fidelity function, and $A_j$ linear operators acting on $u$. The values $A_j u$ are penalised in the total variation or Radon norm $\radonj{\mu} = \radon[\SPACEj]{\mu}$, 
and combined constitute the image regulariser. In this context, $\alpha$ represents the regularisation parameter that balances the strength of regularisation against the fitness $\Phi$ of the solution to the idealised forward model $K$. The size of this parameter depends on the level of random noise and the properties of the forward operator. Choosing it too large results in over-regularisation of the solution and in turn may cause the loss of potentially important details in the image; choosing it too small under-regularises the solution and may result in a noisy and unstable output. In this work we will discuss and thoroughly analyse a bilevel optimisation approach that is able to determine the optimal choice of $\alpha$ in $\Jorig[\alpha]{}$. 

Recently, bilevel approaches for variational models have gained increasing attention in image processing and inverse problems in general. Based on prior knowledge of the problem in terms of a training set of image data and corresponding model solutions or knowledge of other model determinants such as the noise level, optimal reconstruction models are conceived by minimising a cost functional -- called $F$ in the sequel -- constrained to the variational model in question. We will explain this approach in more detail in the next section. Before, let us give an account of the state of the art of bilevel optimisation for model learning. In machine learning bilevel optimisation is well established. It is a semi-supervised learning method that optimally adapts itself to a given dataset of measurements and desirable solutions. In \cite{tappen2007utilizing,tappen2007learning,domke2012generic,domke2013learning,Chen2013,chen2014}, for instance the authors consider bilevel optimization for finite dimensional Markov random field (MRF) models. In inverse problems the optimal inversion and experimental acquisition setup is discussed in the context of optimal model design in works by Haber, Horesh and Tenorio \cite{haber2003learning,haber2008numerical,haber2010numerical}, as well as Ghattas et al. \cite{bui2008model,biegler2011large}.  Recently parameter learning in the context of functional variational regularisation models also entered the image processing community with works by the authors \cite{de2013image,calatronidynamic}, Kunisch, Pock and co-workers \cite{kunisch2013bilevel,Chen2012} and Chung et al. \cite{chung2014optimal}. A very interesting contribution can be found in a preprint by Fehrenbach et al. \cite{fehrenbachbilevel} where the authors determine an optimal regularisation procedure introducing particular knowledge of the noise distribution into the learning approach.

Apart from the work of the authors \cite{de2013image,calatronidynamic}, all approaches for bilevel learning in image processing so far are formulated and optimised in the discrete setting. Our subsequent modelling, analysis and optimisation will be carried out in function space rather than on a discretisation of the variational model. In this context, a careful analysis of the bilevel problem is of great relevance for its application in image processing. In particular, the structure of optimal regularisers is important, among others, for the development of solution algorithms. In particular, if the parameters are bounded and lie in the interior of a closed connected set, then efficient optimization methods can be used for solving the problem. Previous results on optimal parameters for inverse problems with partial differential equations have been obtained in, e.g., \cite{Chavent}.

In this paper we study the qualitative structure of regularization parameters arising as solution of bilevel optimisation problem of variational models. In our framework the variational models are typically convex but non-smooth and posed in Banach spaces. The total variation and total generalized variation regularisation models are particular instances. Alongside the optimisation of the non-smooth variational model, we also consider a smoothed approximation in Hilbert space which is typically the one considered in numerical computation. Under suitable conditions, we prove that -- for both the original non-smooth optimisation problem as well as the regularised Hilbert space problem -- the optimal regularisers are bounded and lie in the interior of the positive orthant. The conditions necessary to prove this turn out to be very natural conditions on the given data in the case of an $L^2$ cost functional $F$. Indeed, for the total variation regularisers with $L^2$-squared cost and fidelity, we will merely require
\[
    \TV(f) > \TV(f_0),
\]
with $f_0$ the ground-truth and $f$ the noisy image. That is, the noisy image should oscillate more in terms of the total variation functional, than the ground-truth. For second-order total generalised variation \cite{bredies2011tgv}, we obtain an analogous condition.
Apart from the standard $L^2$ costs, we also discuss costs that constitute a smoothed $L^1$ norm of the gradient of the original data -- we will call this the Huberised total variation cost in the sequel -- typically resulting in optimal solutions superior to the ones minimising an $L^2$ cost. For this case, however, the interior property of optimal parameters could be verified for a finite dimensional version of the cost only. Eventually, we also show that as the numerical smoothing vanishes the optimal parameters for the smoothed models tend to optimal parameters of the original model.

The results derived in this paper are motivated by problems in image processing. However, their applicability goes well beyond that and can be generally applied to parameter estimation problems of variational inequalities of the second kind, for instance the parameter estimation problem in Bingham flow \cite{de2010optimization}. Previous analysis in this context either required box constraints on the parameters in order to prove existence of solutions or the addition of a parameter penalty to the cost functional \cite{Barbu1993,Bergounioux1998,BergouniouxMignot2000,delosreyes2014learning}. In this paper, we require neither but rather prove that under certain conditions on the variational model and for reasonable data and cost functional, optimal parameters are indeed positive and guaranteed to be bounded away from $0$ and $\infty$. As we will see later this is enough for proving existence of solutions and continuity of the solution map. The next step from our work in this here is deriving numerically useful characterisations of solutions to the ensuing bi-level programs. For the most basic problems considered herein this has been done in \cite{delosreyes2014learning} under numerical $H^1$ regularisation. We will consider in the follow-up work \cite{tuomov-tgvlearn} the optimality conditions for higher-order regularisers and the new cost functionals introduced in this work. For an extension characterisation of optimality systems for bi-level optimisation in finite dimensions, we point the reader to \cite{dempe2014kkt} as a starting point.

\paragraph{Outline of the paper} In Section \ref{sec:probsetup} we introduce the general bilevel learning problem, stating assumptions on the setup of the cost functional $F$ and the lower level problem given by a variational regularisation approach. The bilevel problem is discussed in its original non-smooth form \eqref{eq:learn} as well as in a smoothed form in a Hilbert space setting \eqref{eq:learn-numerical-single} in Section \ref{sec:numerical-considerations}, that will be the one used in the numerical computations. The bilevel problem is put in context with parameter learning for non-smooth variational regularisation models, typical in image processing, by proving the validity of the assumptions for examples such as $\TV$, $\TGV$ and ICTV regularisation. The main results of the paper -- existence of positive optimal parameters for $L^2$, Huberised TV and $L^1$ type costs and the convergence of the smoothed numerical problem to the original non-smooth problem -- are stated in Section \ref{sec:main}. Auxiliary results, such as coercivity, lower semicontinuity and compactness results for the involved functionals, is the topic of Section \ref{sec:auxres}. Proofs for existence and convergence of optimal parameters are contained in Section \ref{sec:proof}. The paper finishes with a brief numerical discussion in Section \ref{sec:numerics}. 

\section{The general problem setup} \label{sec:probsetup}

Let $\Omega \subset \R^m$ be an open bounded domain with Lipschitz boundary. This will be our image domain. Usually $\Omega=(0, w) \times (0, h)$ for $w$ and $h$ the width and height of a two-dimensional image, although no such assumptions are made in this work. Our noisy or corrupted data $f$ is assumed to lie in a Banach space $Y$, which is the dual of $Y_*$, while our ground-truth $f_0$ lies in a general Banach space $Z$. Usually, in our model, we choose $Z=Y=L^2(\Omega; \R^\DIMkurange)$. This holds for denoising or deblurring, for example, where the data is just a corrupted version of the original image. Further $d=1$ for grayscale images, and $d=3$ for colour images in typical colour spaces. For sub-sampled reconstruction from Fourier samples, we might use a finite-dimensional space $Y=\C^n$ -- there are however some subtleties with that, and we refer the interested reader to \cite{adcock2014generalized}.

As our parameter space for regularisation functional weights we take
\[
    \SPACEalphaCompactPos \defeq \SPACEalphaRC,
\]
where $N$ is the dimension of the parameter space, that is $\alpha=(\alpha_1,\ldots,\alpha_N)$, $N\geq 1$. Observe that we allow infinite and zero values for $\alpha$. The reason for the former is that in case of $\TGV^2$, it is not reasonable to expect that \emph{both} $\alpha_1$ and $\alpha_2$ are bounded; such conditions would imply that $\TGV^2$ performs better than both $\TV^2$ and $\TV$.  But we want our learning algorithm to find out whether that is the case! Regarding zero values, one of our main tasks is proving that for reasonable data, optimal parameters in fact lie in the interior
\[
    \SPACEalphaCompactPosInt = \SPACEalphaRCInt.
\]
This is required for the existence of solutions and the continuity of the solution map parametrised by additional regularisation parameters needed for the numerical realisation of the model. We also set
\[
    \SPACEalphaPos \defeq [0, \infty)^\NA
\]
for some occasions when we need a bounded parameter.

\begin{remark}
     In much of our treatment, we could allow for spatially dependent parameters $\alpha$.
     However, the parameters would need to lie in a finite-dimensional subspace of $\SPACEalphaC$ in our theory. Minding our general definition of the functional $\Jhuber[\alpha]{\freevar}$ below, no generality is lost by taking $\alpha$ to be vectors in $\R^\NA$. We could simply replace the sum in the functional as a larger sum modelling integration over parameters with values in a finite-dimensional subspace of $\SPACEalphaC$. 
\end{remark}

In our general learning problem, we look for $\alpha=(\alpha_1, \ldots, \alpha_\NA)$ solving for some convex, proper, weak* lower semicontinuous cost functional $F: X \to \R$ the problem
\begin{equation}
    \label{eq:learn}
    \tag{$\mathrm{P}$}
    \min_{\alpha \in \SPACEalphaCompactPos} \costf(\costK u_{\alpha})
\end{equation}
subject to
\begin{equation}
    \label{eq:denoise}
    \tag{$\mathrm{D}_\alpha$}
    u_{\alpha} \in \argmin_{u \in X} \Jorig[\alpha]{u},
\end{equation}
with
\[
    \Jorig[\alpha]{u} \defeq
        \Phi(Ku)
        +
        \sum_{j=1}^{\NA} \alpha_j \radonj{A_j u}.
\]
Here we denote for short the total variation norm
\[
    \radonj{\mu} \defeq \radon[\SPACEj]{\mu}, \quad (\mu \in \SPACEj).
\]
The following covers our assumptions with regard to $A$, $K$, and $\Phi$.
We discuss various specific examples in Section \ref{sec:example} immediately after stating the assumptions.

\begin{namedassumption}[Operators $A$ and $K$]{$\mbox{KA}$}
    \label{ass:a-k}
    We assume that $Y$ is Banach spaces, and $X$ a normed linear space, both themselves duals of $Y_*$ and $X_*$, respectively. 
    We then assume that the linear operators
    \[
        A_j: X \to \SPACEj,
        \quad
        (j=1,\ldots,\NA),
    \]
    and
    \[
        K: X \to Y,
    \]
    are bounded. Regarding $K$, we also assume the existence of a bounded a right-inverse $\invK: \range{K} \to X$, where $\range{K}$ denotes the range of the operator $K$. That is, $K \invK=I$ on $\range{K}$. 
    We further assume that
    \begin{equation}
        \label{eq:X-coercive}
        \norm{u}_{X}'
        \defeq
        \sum_{j=1}^{\NA} \radonj{A_j u} + \norm{Ku}_Y
    \end{equation}
    is a norm on $X$, equivalent to the standard norm. 
    In particular, by the Banach--Alaoglu theorem, any sequence $\{u^i\}_{i=1}^\infty \subset X$ with $\sup_i \norm{u^i}_{X}' < \infty$ possesses a weakly*
    convergent subsequence.
\end{namedassumption}

\begin{namedassumption}[The fidelity $\Phi$]{$\Phi$}
    \label{ass:phi}
    We suppose $\Phi: Y \to (-\infty, \infty]$ is convex, proper, weakly* lower semicontinuous, and coercive in the sense that 
    \begin{equation}
        \label{eq:Y-coercive}
        \Phi(v) \to +\infty \quad\text{as}\quad \norm{v}_Y \to +\infty .   \end{equation}
    We assume that $0 \in \Dom \Phi$, and the existence of $f \in \argmin_{v \in Y} \Phi(v)$ such that $f=K\bar f$ for some $\bar f \in X$.
    Finally, we require that either $K$ is compact, or $\Phi$ is continuous and strongly convex.
\end{namedassumption}

\begin{remark}
    \label{remark:barf}
    When $\bar f$ exists, we can choose $\bar f=\invK f$.
\end{remark}

\begin{remark}
    Instead of $0 \in \Dom \Phi$, it would suffice to assume, more generally, that $\Dom \Phi \isect \Isect_{j=1}^\NA \ker A_j \ne \emptyset$.
\end{remark}


We also require the following technical assumption on the relationship of the regularisation terms and the fidelity $\Phi$. Roughly speaking, in most interesting cases, it says that for each $\ell$, we can closely approximate the noisy data $f$ with functions of order $\ell$. But this is in a lifted sense, not directly in terms of derivatives.

\begin{namedassumption}[Order reduction]{$\delta$}
    \label{ass:aj-zeroterm-approx}
    We assume that for every $\ell \in \{1,\ldots, \NA\}$ and $\delta>0$,
    there exists $\bar f_{\delta,\ell} \in X$ such that
    \begin{subequations}
    \label{eq:f-approx}
    \begin{align}
        \label{eq:f-approx-delta}
        \Phi(K\bar f_{\delta,\ell})
        &
        < \delta + \inf \Phi,
        \\
        \radonj[\ell]{A_\ell \bar f_{\delta,\ell}}
        &
        < \infty,
        \quad\text{and},
        \\
        \label{eq:f-approx-zero}
        \sum_{j \ne \ell} \radonj{A_j \bar f_{\delta,\ell}}
        &
        = 0.
    \end{align}
    \end{subequations}
\end{namedassumption}

\subsection{Specific examples}
\label{sec:example}

We now discuss a few examples to motivate the abstract framework above.

\begin{example}[Squared $L^2$ fidelity]
    \label{ex:sql2}
    With $Y=L^2(\Omega)$, $f \in Y$, and
    \[
        \Phi(v)=\frac{1}{2}\norm{v-f}_Y^2,
    \]
    we recover the standard $L^2$-squared fidelity, modelling Gaussian noise. On a bounded domain $\Omega$, Assumption \ref{ass:phi} follows immediately.
\end{example}

\begin{example}[Total variation denoising]
    \label{example:tv-1}
    Let us take $K_0$ as the embedding of $X=\BVspace(\Omega) \isect L^2(\Omega)$ into $Z=L^2(\Omega)$, and $A_1=D$. We equip $X$ with the norm
    \[
        \norm{u}_X \defeq \norm{u}_{L^2(\Omega)} + \norm{Du}_{\Meas(\Omega; \R^\DIMdomain)}.
    \]
    This makes $K_0$ a bounded linear operator.
    If the domain $\Omega$ has Lipschitz boundary, and the dimension satisfies $\DIMdomain \in \{1,2\}$, the space $\BVspace(\Omega)$ continuously embeds into $L^2(\Omega)$ \cite[Corollary 3.49]{ambrosio2000fbv}. Therefore, we may identify $X$ with $\BVspace(\Omega)$ as a normed space. Otherwise, if $\DIMdomain \ge 3$, without going into the details of constructing $X$ as a dual space,\footnote{This can be achieved by allowing $\phi_0 \in L^2(\Omega)$ instead of the $C_0(\Omega)$ in the construction of \cite[Remark 3.12]{ambrosio2000fbv}.} we \emph{define} weak* convergence in $X$ as combined weak* convergence in $\BVspace(\Omega)$ and $L^2(\Omega)$.  Any bounded sequence in $X$ will then have a weak* convergent subsequence. This is the only property we would use from $X$ being a dual space.

    Now, combined with Example \ref{ex:sql2} and the choice $K=K_0$, $Y=Z$, we get total variation denoising for the sub-problem \eqref{eq:denoise}. Assumption \ref{ass:a-k} holding is immediate from the previous discussion.
    Assumption \ref{ass:aj-zeroterm-approx} is also easily satisfied, as with $f \in \BVspace(\Omega) \isect L^2(\Omega)$, we may simply pick $\bar f_{\delta,1}=f$ for the only possible case $\ell=1$.
    Observe however that $K$ is not compact unless $\DIMdomain=1$, see \cite[Corollary 3.49]{ambrosio2000fbv}, so the strong convexity of $\Phi$ is crucial here. If the data $f \in L^\infty(\Omega)$, then it is well-known that solutions $\hat u$ to \eqref{eq:denoise} satisfy $\norm{\hat u}_{L^\infty(\Omega)} \le \norm{f}_{L^\infty(\Omega)}$. We could therefore construct a compact embedding by adding some artificial constraints to the data $f$. This changes in the next two examples, as boundedness of solutions for higher-order regularisers is unknown; see also \cite{tuomov-jumpset2}.
\end{example}

\begin{example}[Second order total generalised variation denoising]
    \label{example:tgv2-1}
    We take
    \[
        X=(\BVspace(\Omega) \isect L^2(\Omega)) \times \BDspace(\Omega),
    \]
    the first part with the same topology as in Example \ref{example:tv-1}. We also take $Z=L^2(\Omega)$, denote $u=(v, w)$, and set
    \[
        K_0(v, w)=v, \quad A_1u = Dv-w,\quad \text{ and } \quad A_2u =Ew
    \]
    for $E$ the symmetrised differential. With $K=K_0$ and $Y=Z$, this yields second-order total generalised variation ($\TGV^2$) denoising \cite{bredies2011tgv} for the sub-problem \eqref{eq:denoise}. 
    Assuming for simplicity that $\alpha_1, \alpha_2 > 0$ are constants,
    to show Assumption \ref{ass:a-k}, we recall from \cite{sampta2011tgv} 
    the existence of a constant $c=c(\Omega)$ such that
    \[
        \inv c\norm{v}_{\BVspace(\Omega)}
        \le
        \norm{v}_{L^1(\Omega)}
        +
        \norm{Dv}_{\Meas(\Omega; \R^\DIMdomain)}
        \le
        c\norm{v}_{\BVspace(\Omega)},
    \]
    where the norm
    \[
        \norm{v}_{\BVspace(\Omega)} 
        \defeq
        \TGV^2_{(1,1)}(v) + \norm{v}_{L^1(\Omega)}.
    \]
    We may thus approximate
    \[
        \begin{split}
        \norm{w}_{L^1(\Omega; \R^\DIMdomain)}
        &
        \le
        \norm{Dv-w}_{\Meas(\Omega; \R^\DIMdomain)} + \norm{Dv}_{\Meas(\Omega; \R^\DIMdomain)}
        \\
        &
        \le
        \norm{Dv-w}_{\Meas(\Omega; \R^\DIMdomain)} + c\bigl(\TGV^2_{(1,1)}(v) + \norm{u}_{L^1(\Omega)}\bigr),
        \\
        &
        \le
        (1+c)\bigl(\norm{Dv-w}_{\Meas(\Omega; \R^\DIMdomain)} + \norm{Ew}_{\Meas(\Omega; \R^{\DIMdomain \times \DIMdomain})}\bigr)
        +
        c \norm{v}_{L^1(\Omega)}.
        \end{split}
    \]
    For some $C>0$, it follows
    \[
        \begin{split}
        \norm{u}_{X}
        & =
        \left(
        \norm{v}_{L^2(\Omega)}
        +
        \norm{Dv}_{\Meas(\Omega; \R^\DIMdomain)}
        \right)
        +
        \left(
        \norm{w}_{L^1(\Omega; \R^\DIMdomain)}
        +
        \norm{Ew}_{\Meas(\Omega; \R^{\DIMdomain \times \DIMdomain})}
        \right)
        \\
        &
        \le
        C\left(\norm{Dv-w}_{\Meas(\Omega; \R^\DIMdomain)} + \norm{Ew}_{\Meas(\Omega; \R^{\DIMdomain \times \DIMdomain})}
        +
        \norm{v}_{L^2(\Omega)}\right)
        \\
        &
        =
        C\left(\sum_{i=1}^\NA \radonj{A_j u}
        +
        \norm{Ku}_{L^2(\Omega)}\right).
        \end{split}
    \]
    This shows $\norm{u}_{X} \le C\norm{u}_{X}'$. 
    The inequality $\norm{u}_{X} \ge \norm{u}_{X}'$
    follows easily from the triangle inequality, namely
    \[
        \norm{u}_{X}
        \ge
        \norm{v}_{L^2(\Omega)}
        +
        \norm{Dv-w}_{\Meas(\Omega; \R^\DIMdomain)}
        +
        \norm{Ew}_{\Meas(\Omega; \R^{\DIMdomain \times \DIMdomain})}.
    \]
    Thus $\norm{\freevar}_{X}'$ is equivalent to $\norm{\freevar}_{X}$.

    Next we observe that clearly $Dv^k-w^k \weaktostar Dv-w$ in $\Meas(\Omega; \R^\DIMdomain)$ 
    and $Ew^k \weaktostar Ev$ in $\Meas(\Omega; \R^{\DIMdomain \times \DIMdomain})$ 
    if $v^k \weaktostar v$ in $\BVspace(\Omega)$ and $w^k \weaktostar w$ weakly* in $\BDspace(\Omega)$.
    Thus $A_1$ and $A_2$ are weak* continuous. Assumption \ref{ass:a-k} follows.

    Consider then the satisfaction of Assumption \ref{ass:aj-zeroterm-approx}. If $\ell=1$, we may then pick $\bar f_{\delta,1}=(f, 0)$, in which case $A_1 f_{\delta,1}=Df$, and $A_2 f_{\delta,1} = 0$. If $\ell=2$, which is the only other case, we may pick a smooth approximate $f_{\delta,2}$ to $f$ such that
    \[
        \frac{1}{2} \norm{f-f_{\delta,2}}_{L^2(\Omega)}^2 < \delta.
    \]
    Then we set $\bar f_{\delta,2} \defeq (f_{\delta,2}, \grad f_{\delta,2})$, yielding $A_1 \bar f_{\delta,2}=0$ and $A_2 \bar f_{\delta,2} = \grad^2 f_{\delta,2}$. By $f_{\delta,2}$ being smooth we see that $\radon{A_2 \bar f_{\delta,2}} < \infty$. Thus Assumption \ref{ass:aj-zeroterm-approx} is satisfied for the squared $L^2$ fidelity $\Phi(v)=\norm{f-v}_{L^2(\Omega)}^2$.
\end{example}

\begin{example}[Infimal convolution TV denoising]
    \label{example:ictv-1}
    Let us take $Z=L^2(\Omega)$, and $X=(\BVspace(\Omega) \isect L^2(\Omega)) \times X_2$, where
    \[
        X_2 \defeq \{ v \in W^{1,1}(\Omega) \mid \grad v \in \BVspace(\Omega; \R^\DIMdomain)\},
    \]
    and $\BVspace(\Omega) \isect L^2(\Omega)$ again has the topology of Example \ref{example:tv-1}.
    Setting $u=(v, w)$, as well as
    \[
        K_0(v, w)=v+w, \quad A_1u = Dv,\quad \text{ and } \quad A_2u =D\grad w,
    \]
    we obtain for \eqref{eq:denoise} with $K=K_0$ and $Y=Z$ the infimal convolution total variation denoising model of \cite{chambolle97image}. Assumption \ref{ass:a-k}, \ref{ass:aj-zeroterm-approx}, and \ref{ass:sol-smoothness}
    are verified analogously to $\TGV^2$ in Example \ref{example:tgv2-1}.
\end{example}

\begin{example}[Cost functionals]
    \label{ex:cost}
    For the cost functional $\costf$, given noise-free data $f_0 \in Z = L^2(\Omega)$, we consider in particular the $L^2$ cost
    \[
        \costf_{\costltwo}(u) \defeq \frac{1}{2}\norm{f_0 - K_0 u}_{L^2(\Omega)}^2, 
    \]
    as well as the Huberised total variation cost
    \[
        \costf_{\costhubertv}(u) \defeq \huberradon[\gamma]{D(f_0- K_0 u)}
    \]
    with noise-free data $f_0 \in Y \defeq \BVspace(\Omega)$. For the definition of the Huberised total variation, we refer to the Section \ref{sec:numerical-considerations} on the numerics of the bi-level framework \eqref{eq:learn}.
\end{example}

\begin{example}[Sub-sampled Fourier transforms for MRI]
    Let $K_0$ and the $A_j$s be given by one of the regularisers of Example \ref{example:tv-1} to \ref{example:ictv-1}. Also take the cost $F = \costf_{\costltwo}$ or $\costf_{\costhubertv}$ as in Example \ref{ex:cost}, and $\Phi$ as the squared $L^2$ fidelity of Example \ref{ex:sql2}.
    However, let us now take $K=TK_0$ for some bounded linear operator $T: Z \to Y$.
    The operator $T$ could be, for example, a blurring kernel or a (sub-sampled) Fourier transform, in which case we obtain a model for learning the parameters for deblurring or recontruction from Fourier samples. The latter would be important, for example for magnetic resonance imaging (MRI) \cite{benning2014phase,tuomov-nlpdhgm,tuomov-dtireg}.
    Unfortunately, our theory does no extend to many of these cases because we will require, roughly, $K_0^*K_0 \le C K^*K$ for some constant $C>0$.
\end{example}

\begin{example}[Parameter estimation in Bingham flow]
\label{ex:bingham}
Bingham fluids are materials that behave as solids if the magnitude of the stress tensor stays below a plasticity threshold, and as liquids if that quantity surpasses the threshold. In a cross sectional pipe, of section $\Omega$, the behaviour is modeled by the energy minimization functional
\begin{equation} \label{eq: bingham functional}
\min_{u \in H_0^1(\Omega)} ~\frac{\mu}{2} \|u\|^2_{H_0^1(\Omega)}- (f | u )_{H^{-1}(\Omega),H_0^1(\Omega)} + \alpha \int_{\Omega} |\nabla u|~dx,
\end{equation}
where $\mu >0$ stands for the viscosity coefficient, $\alpha >0$ for the plasticity threshold and $f \in H^{-1}(\Omega)$. In many practical situations, the plasticity threshold is not known in advance and has to be estimated from experimental measurements. One then aims at minimizing a least squares term 
$$F_{L^2_2}= \frac{1}{2} \|u- f_0\|^2_{L^2(\Omega)}$$
subject to \eqref{eq: bingham functional}.

The bilevel optimization problem can then be formulated as problem \eqref{eq:learn}-\eqref{eq:denoise}, with the choices $X=Y=H_0^1(\Omega)$, $K$ the identity, $A_1=D$ and
$$\phi(v)=\frac{\mu}{2} \|v\|^2_{H_0^1(\Omega)}- (f | u)_{H^{-1}(\Omega),H_0^1(\Omega)}.$$

Concentrating in the rest of this paper primarily on image processing applications, we will however briefly return to Bingham flow in Example \ref{ex:bingham-condition}.
\end{example}

\subsection{Considerations for numerical implementation}
\label{sec:numerical-considerations}
For the numerical solution of the denoising sub-problem, we will in a follow-up work \cite{tuomov-tgvlearn} expand upon the infeasible semi-smooth quasi-Newton approach taken in \cite{hintermuller2006infeasible} for $L^2$-$\TV$ image restoration problems. This depends on Huber-regularisation of the total variation measures, as well as enforcing smoothness through Hilbert spaces. This is usually done by a squared penalty on the gradient, i.e., $H^1$ regularisation, but we formalise this more abstractly in order to simplify our notation and arguments later on.
Therefore, we take a convex, proper, and weak* lower-semicontinous smoothing functional $\Smoother: X \to [0,\infty]$, and generally expect it to satisfy the following.

\begin{namedassumption}[Smoothing]{$\Smoother$}
    \label{ass:sol-smoothness}
    We assume that $0 \in \Dom \Smoother$ and for every $\delta \ge 0$, every
    $\alpha \in \SPACEalphaCompactPos$, and every $u \in X$,
    the existence of $u^\delta \in X$ satisfying
    \begin{equation}
        \label{eq:sol-smoothness}
        \Smoother(u^\delta)<\infty
        \quad\text{and}\quad
        \Jorig[\alpha]{u^\delta} \le \Jorig[\alpha]{u} + \delta.
    \end{equation}
\end{namedassumption}

\begin{example}[$H^1$ smoothing in $\BVspace(\Omega)$]
    Usually, with $\SPACEuHilbert \isect X \ne \emptyset$, we take
    \[
        \Smoother(u) \defeq
        \begin{cases}
            \frac{1}{2} \norm{\grad u}_{\SPACEgraduLtwo}^2, & u \in \SPACEuHilbert \isect X, \\
            \infty, & \text{otherwise}.
        \end{cases}
    \]
    This is in particular the case with Example \ref{example:tv-1} ($\TV$), where
    $X=\BVspace(\Omega) \isect L^2(\Omega) \supset H^1(\Omega)$, and Example \ref{example:tgv2-1} ($\TGV^2$),
    where $X=(\BVspace(\Omega) \isect L^2(\Omega)) \times \BDspace(\Omega) \supset H^1(\Omega) \times H^1(\Omega; \R^\DIMdomain)$ on a bounded domain $\Omega$.
    In both of these cases, weak* lower semicontinuity is apparent; for completeness we record this in Lemma \ref{lemma:lsc} in Section \ref{sec:lsc}.
    In case of Example \ref{example:tv-1}, \eqref{eq:sol-smoothness} is immediate from approximating $u$ strictly by functions in $C^\infty(\Omega)$ using standard strict approximation results in $\BVspace(\Omega)$ \cite{ambrosio2000fbv}. In case of Example \ref{example:tgv2-1}, this also follows by a simple generalisation of the same argument to TGV-strict approximation, as presented in \cite{tuomov-jumpset2,bredies2013regularization}.
\end{example}

For parameters $\epsilon \ge 0$ and $\gamma \in (0, \infty]$, we then consider the problem
\begin{equation}
    \label{eq:learn-numerical-single}
    \tag{$\mathrm{P}^{\gamma,\epsilon}$}
    \min_{\alpha \in \SPACEalphaCompactPos} \costf(\costK u_{\alpha,\gamma,\epsilon})
\end{equation}
where $u_{\alpha,\gamma,\epsilon} \in X \isect \Dom \epsilon\Smoother$ solves
\begin{equation}
    \label{eq:denoise-numerical-single}
    \tag{$\mathrm{D}^{\gamma,\epsilon}$}
    \min_{u \in X} \Jepsilon[\alpha]{u}
\end{equation}
for
\begin{equation}
    \notag
    \Jepsilon[\alpha]{u} :=
        \epsilon \Smoother(u)
        +
        \Phi(Ku)
        +
        \sum_{j=1}^{\NA} \alpha_j \huberj{A_j u}.
\end{equation}
Here we denote for short the Huber-regularised total variation norm
\[
    \huberj{\mu} \defeq \huberradon[\gamma,\SPACEj]{\mu}, \quad (\mu \in \SPACEj),
\]
as given by the following definition.
There we interpret $\gamma=\infty$ to give back the standard unregularised total variation measure.
Clearly $\JepsilonXXSYM{\infty,0}=J$, and \eqref{eq:denoise} corresponds to \eqref{eq:denoise-numerical-single} and \eqref{eq:learn} to \eqref{eq:learn-numerical-single} with $(\gamma,\epsilon)=(\infty,0)$.

\begin{definition}\label{def:huber}
Given $\gamma \in (0, \infty]$, we  define for the norm $\norm{\freevar}_2$ on $\R^m$, 
the Huber regularisation
\[
    \huber[\gamma]{g} = 
    \begin{cases}
        \norm{g}_2 - \frac{1}{2\gamma}, & \norm{g}_2 \ge 1/\gamma,
        \\
        \frac{\gamma}{2}\norm{g}_2^2, & \norm{g}_2 < 1/\gamma.
    \end{cases}
\]
We observe that this can equivalently be written using convex conjugates as
\begin{equation}
    \label{eq:huber-twonorm}
    \alpha \huber[\gamma]{g}=
    \sup
    \Bigl\{
        \iprod{q}{g} - \frac{1}{2\gamma\alpha} \norm{q}_2^2
        \Bigm|
        \norm{q}_2 \le \alpha
    \Bigr\}.
\end{equation}
Then if $\mu = f \L^\DIMdomain + \mu^s$ is the Lebesgue decomposition of $\mu \in \Meas(\Omega; \R^m)$ into the absolutely continuous part $f\L^\DIMdomain$ and the singular part $\mu^s$, we set
\[
    \huber[\gamma]{\mu}(V) \defeq \int_V \huber[\gamma]{f(x)} \d x + \abs{\mu^s}(V),
    \quad
    (V \in \borel(\Omega)).
\]
The measures $\huber[\gamma]{\mu}$ is the Huber-regularisation of the total variation measures $\abs{\mu}$, and we define its Radon norm as the Huber regularisation of the Radon norm of $\mu$, that is
\[
    \huberradon[\gamma,\Meas(\Omega; \R^m)]{\mu} \defeq \radon[\Meas(\Omega; \R^m)]{\huber[\gamma]{\mu}}.
\]
\end{definition}

\begin{remark}
    The parameter $\gamma$ varies in the literature.
    In this paper, we use the convention of \cite{delosreyes2014learning},
    where large $\gamma$ means small Huber regularisation. 
    In \cite{hintermuller2006infeasible}, small $\gamma$ means small
    Huber regularisation. That is, their regularisation parameter is
    $1/\gamma$ in our notation.
\end{remark}

\subsection{Shorthand notation}

Writing for notational lightness
\[
    A u \defeq (A_1 u, \ldots, A_\NA u),
\]
and
\[
    \regf{\mu_1, \ldots, \mu_\NA} \defeq \sum_{j=1}^\NA \alpha_j \huberj{\mu_j},
    \quad
    \regforig{\freevar} \defeq \regfX{\freevar}{\infty},
\]
our problem \eqref{eq:learn-numerical-single} becomes
\begin{equation}
    \notag
    \min_{\alpha \in \SPACEalphaCompactPos} \costf(\costK u_{\alpha})
    \quad\text{subject to}\quad
    u_{\alpha} \in \argmin_{u \in X} \Jepsilon[\alpha]{u}
\end{equation}
for 
\[
    \Jepsilon[\alpha]{u} \defeq
        \epsilon \Smoother(u)
        + \Phi(Ku)
        + \regf{A u}.
\]

Further, given $\alpha \in \SPACEalphaPos$, we define the ``marginalised'' 
regularisation functional
\begin{equation}
    \label{eq:regfmarg}
    \regfMARGhuber[\alpha]{v} \defeq \inf_{\bar v \in \inv Kv} \regf[\alpha]{A\bar v},
    \quad
    \regfMARG{\freevar} \defeq \regfMARGhuberX{\freevar}{\infty}.
\end{equation}
Here $\inv K$ stands for the preimage, so the constraint is $\bar v \in X$ with $v=K\bar v$.
Then in the case $(\gamma, \epsilon)=((\infty,0)$ and $\costf=\costf_0 \circ K$, our problem may also be written as
\[
    \notag
    \min_{\alpha \in \SPACEalphaCompactPos}
        \costf_0(v_{\alpha})
\]
subject to
\[
    v_{\alpha} \in \argmin_{v \in Y}
    \Phi(v) + \regfMARG[\alpha]{v}.
\]
This gives the problem a much more conventional flair, as the following examples demonstrate.

\begin{example}[$\TV$ as a marginal]
    \label{example:tv-marginal}
    Consider the total variation regularisation of Example \ref{example:tv-1}. Then $A_1=D$, $\NA=1$, and
    \[
        \regfMARG[\alpha]{v}=\regforig[\alpha]{A v}=\alpha\TV(v).
    \]
\end{example}

\begin{example}[$\TGV^2$ as a marginal]
    \label{example:tgv2-marginal}
    In case of the $\TGV^2$ regularisation of Example \ref{example:tgv2-1}, we have $K(v,w)=v$ and $\invK v=(v, 0)$. Thus $K\invK f=f$, etc., so
    \[
        \regfMARG[\alpha]{v}=\inf_{w \in \BDspace(\Omega)} \regforig[\alpha]{Dv-w,Ew} = \TGV^2_{\alpha}(v).
    \]
\end{example}

\section{Main results}
\label{sec:main}

Our task now is to study the characteristics of optimal solutions, and their existence. Our results, based on natural assumptions on the data and the original problem \eqref{eq:learn}, derive properties of the solutions to this problem and all numerically regularised problems \eqref{eq:learn-numerical-single} sufficiently close to the original problem: large $\gamma>0$ and small $\epsilon>0$.
We denote by $u_{\alpha,\gamma,\epsilon}$ any solution to \eqref{eq:denoise-numerical-single} for any given $\alpha \in \SPACEalphaPos$, and by $\alpha_{\gamma,\epsilon}$ any solution to \eqref{eq:learn-numerical-single}. Solutions to \eqref{eq:denoise} and \eqref{eq:learn} we denote, respectively, by $u_\alpha=u_{\alpha,\infty,0}$, and $\hat \alpha=\alpha_{\infty,0}$. 

\subsection{$L^2$-squared cost and $L^2$-squared fidelity}

Our main existence result regarding $L^2$-squared costs and $L^2$-squared fidelities is the following.

\begin{theorem}
    \label{thm:l2cost-main}
    Let $Y$ and $Z$ be Hilbert spaces, $\costf(u)=\frac{1}{2}\norm{K_0 u - f_0}_Z^2$, and $\Phi(v)=\frac{1}{2}\norm{f-v}_Y^2$ for some $f \in \range{K}$, $f_0 \in Z$, and a bounded linear operator $K_0: X \to Z$ satisfying
    \begin{equation}
        \label{eq:k0-condition}
	\norm{K_0 u}_Z \le C_0 \norm{Ku}_Y, \quad \text{for all } u \in X
        \quad \text{for some constant $C_0>0$}.
    \end{equation}
    Suppose Assumption \ref{ass:a-k} and \ref{ass:aj-zeroterm-approx} hold.
    If for some $\bar\alpha \in \SPACEalphaPosInt$ and $t \in (0, 1/C_0]$ holds
    \begin{equation}
        \label{eq:l2cost-interior-condition-k0}
        \regfMARG{f} > \regfMARG{f-t(K_0\invK)^*(K_0 \invK f-f_0)},
    \end{equation}
    then problem  \eqref{eq:learn} admits a solution $\hat \alpha \in \SPACEalphaCompactPosInt$.

    If, moreover, \ref{ass:sol-smoothness} holds, then there exist $\bar\gamma \in (0, \infty)$ and $\bar \epsilon \in (0, \infty)$ such that the problem \eqref{eq:learn-numerical-single} with $(\gamma, \epsilon) \in [\bar\gamma,\infty] \times [0, \bar\epsilon]$ admits a solution $\alpha_{\gamma,\epsilon} \in \SPACEalphaCompactPosInt$, and the solution map 
    \[
        \SolM(\gamma, \epsilon) \defeq \argmin_{\alpha \in \SPACEalphaCompactPos}~\costf(u_{\alpha,\gamma,\epsilon})
    \]
    is outer semicontinuous within $[\bar\gamma,\infty] \times [0, \bar\epsilon]$.
\end{theorem}

We prove this result in Section \ref{sec:proof}. Outer semicontinuity of a set-valued map $S: \R^k \setto \R^m$ means \cite{rockafellar-wets-va} that for any convergent sequence $x^k \to x$ and $S(x^k) \ni y^k \to y$, we have $y \in S(x)$. In particular, the outer semicontinuity of $\SolM$ means that as the numerical regularisation vanishes, the optimal parameters for the regularised models \eqref{eq:learn-numerical-single} tend to optimal parameters of the original model \eqref{eq:learn}.

\begin{remark}
    \label{remark:l2cost-interior-condition}
    Let $Z=Y$ and $K_0=K$ in Theorem \ref{thm:l2cost-main}. Then \eqref{eq:l2cost-interior-condition-k0}
    reduces to
    \begin{equation}
        \label{eq:l2cost-interior-condition}
        \regfMARG{f} > \regfMARG{f_0}.
    \end{equation}
    Also observe that our result requires, $\Phi \circ K$ to measure all the data that $F$ measures, in the more precise sense given by \eqref{eq:k0-condition}.
    If \eqref{eq:k0-condition} did not hold, an oscillating solution $u_\alpha$ for $\alpha \in \BD \SPACEalphaCompactPos$, could largely pass through the nullspace of $K$, hence have low value for the objective $J$ of the inner problem, yet have a large cost given by $F$.
\end{remark}

\begin{corollary}[Total variation Gaussian denoising]
    \label{corollary:tv-l2cost-interior}
    Suppose $f, f_0 \in \BVspace(\Omega) \isect L^2(\Omega)$, and
    \begin{equation}
        \label{eq:l2cost-interior-condition-tv}
        \TV(f)
        >
        \TV(f_0).
    \end{equation}
    Then there exist $\bar \epsilon, \bar\gamma > 0$ such that any optimal solution $\alpha_{\gamma,\epsilon}$
    to the problem
    \[
        \alpha_{\gamma,\epsilon} \in \argmin_{\alpha \ge 0} \frac{1}{2}\norm{f_0-u_\alpha,\gamma,\epsilon}_{L^2(\Omega)}^2
    \]
    with
    \[
        u_{\alpha,\gamma,\epsilon} \in \argmin_{u \in \BVspace(\Omega)} 
            \Bigl(
            \frac{1}{2}\norm{f-u}_{L^2(\Omega)}^2
            + \alpha\huberradon{Du}
            + \frac{\epsilon}{2}\norm{\grad v}_{L^2(\Omega; \R^\DIMdomain)}^2
            \Bigr)
    \]
    satisfies $\alpha_{\gamma,\epsilon} > 0$ whenever $\epsilon \in [0, \bar \epsilon]$, $\gamma \in [\bar\gamma, \infty]$.
\end{corollary}

That is, for the optimal parameter to be strictly positive, the noisy image $f$ should, in terms of the total variation, oscillate more than the noise-free image $f_0$. This is a very natural condition: if the noise somehow had smoothed out features from $f_0$, then we should not smooth it anymore by $\TV$ regularisation!

\begin{proof}
    Assumption \ref{ass:a-k}, \ref{ass:aj-zeroterm-approx}, and \ref{ass:sol-smoothness} we
    have already verified in Example \ref{example:tv-1}.
    We then observe that $K_0=K$, so we are in the setting of Remark \ref{remark:l2cost-interior-condition}.
    Following the mapping of the TV problem to the general framework using the construction in Example \ref{example:tv-1}, we have $K=I$ and $\invK=I$ embeddings with $Y = L^2(\Omega)$.
    $\invK$ is bounded on $\range{K}=L^2(\Omega) \isect \BVspace(\Omega)$.
    Moreover, by Example \ref{example:tv-marginal}, $\regfMARG{v}=\bar\alpha\TV(v)$.
    Thus \eqref{eq:l2cost-interior-condition} with the choice $t=1$ reduces to \eqref{eq:l2cost-interior-condition-tv}.
\end{proof}

For $\TGV^2$ we also have a very natural condition.

\begin{corollary}[Second-order total generalised variation Gaussian denoising]
    \label{corollary:tgv2-l2cost-interior}
    Suppose that the data $f, f_0 \in L^2(\Omega) \isect \BVspace(\Omega)$ satisfies for some $\alpha_2>0$ the condition
    \begin{equation}
        \label{eq:l2cost-interior-condition-tgv2}
        \TGV^2_{(\alpha_2, 1)}(f)
        >
        \TGV^2_{(\alpha_2, 1)}(f_0).
    \end{equation}
    Then there exists $\bar \epsilon,\bar\gamma > 0$ such that any optimal solution $\alpha_{\gamma,\epsilon}=((\alpha_{\gamma,\epsilon})_1, (\alpha_{\gamma,\epsilon})_2)$ to the problem
    \[
        \alpha_{\gamma,\epsilon} \in  \argmin_{\alpha \ge 0} \frac{1}{2}\norm{f_0-v_{\alpha,\gamma,\epsilon}}_{L^2(\Omega)}^2
    \]
    with
    \[
        \begin{split}
        (v_{\alpha,\gamma,\epsilon}, w_{\alpha,\gamma,\epsilon}) \in \argmin_{\substack{v \in \BVspace(\Omega)\\ w \in \BDspace(\Omega)}} 
            \Bigl(
            &
            \frac{1}{2}\norm{f-v}_{L^2(\Omega)}^2
            + \alpha_1\huberradon{Dv-w}
            + \alpha_2\huberradon{Ew}
            \\
            &
            + \frac{\epsilon}{2}\norm{(\grad v, \grad w)}_{L^2(\Omega; \R^\DIMdomain \times \R^{\DIMdomain \times \DIMdomain})}^2
            \Bigr)
        \end{split}
    \]
    satisfies $(\alpha_{\gamma,\epsilon})_1, (\alpha_{\gamma,\epsilon})_2 > 0$ whenever $\epsilon \in [0, \bar \epsilon]$, $\gamma \in [\bar\gamma, \infty]$.
\end{corollary}

\begin{proof}
    Assumption \ref{ass:a-k}, \ref{ass:aj-zeroterm-approx}, and \ref{ass:sol-smoothness} we
    have already verified in Example \ref{example:tgv2-1}.
    We then observe that $K_0=K$, so we are in the setting of Remark \ref{remark:l2cost-interior-condition}. By Example \ref{example:tgv2-marginal}, $\regfMARG{v} = \TGV^2_{\bar\alpha}(v)$.
    Finally, similarly to \eqref{eq:l2cost-interior-condition-tv}, we get for \eqref{eq:l2cost-interior-condition} with $t=1$ the condition \eqref{eq:l2cost-interior-condition-tgv2}.
\end{proof}

\begin{example}[Fourier reconstructions]
    Let $K_0$ be given, for example as constructed in Example \ref{example:tv-1} or Example \ref{example:tgv2-1}.
    If we take $K=\mathcal{F} K_0$ for $\mathcal{F}$ the Fourier transform -- or any other unitary transform -- then \eqref{eq:k0-condition} is satisfied
    and $\invK=\invK_0 \mathcal{F}^*$.
    Thus \eqref{eq:l2cost-interior-condition-k0} becomes
    \begin{equation}
        \notag
        \regfMARG{f} > \regfMARG{f-t(K_0 \invK_0 \mathcal{F}^*)^* (K_0 \invK_0 \mathcal{F}^*  f-f_0) }.
    \end{equation}
    With $\mathcal{F}^* f, f_0 \in \range{K_0}$ and $t=1$ this just reduces to
    \begin{equation}
        \notag
        \regfMARG{f} > \regfMARG{\mathcal{F} f_0}.
    \end{equation}
    
    Unfortunately, our results do not cover parameter learning for reconstruction from partial Fourier samples exactly because of \eqref{eq:k0-condition}. 
    What we can do is to find the optimal parameters if we only know a part of the ground-truth, but have full noisy data. 
\end{example}

\subsection{Huberised total variation and other $L^1$-type costs with $L^2$-squared fidelity}

We now consider the alternative ``Huberised total variation'' cost functional from \ref{ex:cost}. Unfortunately, we are unable to derive for $\costf_{\costhubertv}$ easily interpretable conditions as for the $\costf_{\costltwo}$. If we discretise the definition in the following sense, then we however get natural conditions. So, we let $\tilde f_0 \in Z$, assuming $Z$ is a reflexive Banach space, and pick $\eta \in (0, \infty]$. We define 
\[
    \costf_{\costlone}(z) \defeq \sup_{\norm{\lambda}_{Z^*} \le 1}~ \dprod{\lambda}{z-\tilde f_0} - \frac{1}{2\eta}\norm{\lambda}_{Z^*}^2.
\]
If 
\[
    V = \{ \xi_1, \ldots, \xi_M \} \subset \{ \xi \in C_c^\infty(\Omega; \R^\DIMdomain) \mid \norm{\xi} \le 1\},
\]
is finite-dimensional, we define
\[
    D^V v \defeq \{\dprod{-\divergence \xi_1}{v}, \ldots, \dprod{-\divergence \xi_M}{v} \},
    \quad
    (v \in \BVspace(\Omega)).
\]
We may now approximate $\costf_{\costhubertv}$ by
\[
    \costf_{\costhubertv}^V \defeq \costf_{\costlone} \circ D^V
    \quad
    \text{where } \tilde f_0 \defeq D^V f_0 \text { and }  Z=\R^M \text{ with } \infty\text{-norm},
\]
We return to this approximation after the following general results on $\costf_{\costlone}$.

\begin{theorem}
    \label{thm:hubercost-main}
    Let $Y$ be a Hilbert space, and $Z$ a reflexive Banach space. Let $\costf=\costf_{\costlone} \circ K_0$, and $\Phi(v)=\frac{1}{2}\norm{f-v}_Y^2$
    for some \emph{compact} linear operator $K_0: X \to Z$ satisfying \eqref{eq:k0-condition} and $f \in \range{K}$.
    Suppose Assumption \ref{ass:a-k} and \ref{ass:aj-zeroterm-approx} hold.
    If for some $\bar\alpha \in \SPACEalphaPosInt$ and $t>0$ holds
    \begin{equation}
        \label{eq:hubercost-interior-condition-k0}
        \regfMARG{f} > \regfMARG{f - t(K_0 \invK)^* \lambda},
        \quad
        \lambda \in \subdiff \costf_{\costlone}(K_0 \invK f),
    \end{equation}
    then the problem  \eqref{eq:learn} admits a solution $\hat \alpha \in \SPACEalphaCompactPosInt$.

    If, moreover, \ref{ass:sol-smoothness} holds, then there exists there exist $\bar\gamma \in (0, \infty)$ and $\bar \epsilon \in (0, \infty)$ such that the problem \eqref{eq:learn-numerical-single} with $(\gamma, \epsilon) \in [\bar\gamma,\infty] \times [0, \bar\epsilon]$ admits a solution $\alpha_{\gamma,\epsilon} \in \SPACEalphaCompactPosInt$, and the solution map $\SolM$ is outer semicontinuous within $[\bar\gamma,\infty] \times [0, \bar\epsilon]$.
\end{theorem}

We prove this result in Section \ref{sec:proof}.

\begin{remark}
    \label{remark:hubercost-interior-condition}
    If $K_0=K$, the condition \eqref{eq:hubercost-interior-condition-k0} has the much more legible form
    \begin{equation}
        \label{eq:hubercost-interior-condition}
        \regfMARG{f} > \regfMARG{f-t \lambda},
        \quad
        \lambda \in \subdiff \costf_{\costlone}(f),
    \end{equation}
    Also if $K$ is compact, then the compactness of $K_0$ follows from \eqref{eq:k0-condition}.
    In the following applications, $K$ is however not compact for typical domains $\Omega \subset \R^2$ or $\R^3$, so we have to make $K_0$ compact by making the range finite-dimensional.
\end{remark}

\begin{corollary}[Total variation Gaussian denoising with discretised Huber-TV cost]
    \label{corollary:tv-hubercost-interior}
    Suppose that the data satisfies $f, f_0 \in \BVspace(\Omega) \isect L^2(\Omega)$ and
    for some $t>0$ and $\xi \in V$ the condition
    \begin{equation}
        \label{eq:hubercost-interior-condition-tv}
        \TV(f)
        >
        \TV(f+t\divergence \xi),
        \quad
        - \divergence \xi \in \subdiff \costf_{\costhubertv}^V(f).
    \end{equation}
    Then there exists $\bar \epsilon,\bar\gamma > 0$ such any optimal solution $\alpha_{\gamma,\epsilon}$ to the problem
    \[
        \min_{\alpha \ge 0} \costf_{\costhubertv}^V(f_0-v_\alpha)
    \]
    with
    \[
        u_\alpha \in \argmin_{u \in \BVspace(\Omega)} 
            \Bigl(
            \frac{1}{2}\norm{f-u}_{L^2(\Omega)}^2
            + \alpha\huberradon{Du}
            + \frac{\epsilon}{2}\norm{\grad v}_{L^2(\Omega; \R^\DIMdomain)}^2
            \Bigr)
    \]
    satisfies $\alpha_{\gamma,\epsilon} > 0$ whenever $\epsilon \in [0, \bar \epsilon]$, $\gamma \in [\bar\gamma, \infty]$.
\end{corollary}

This says that for the optimal parameter to be strictly positive, the noisy image $f$ should oscillate more than the image $f+t \divergence \xi$ in  the direction of the \emph{(discrete) total variation flow}. This is a very natural condition, and we observe that the non-discretised counterpart of \eqref{eq:hubercost-interior-condition-tv} for $\gamma=\infty$ would be
\begin{equation}
    \notag
    \TV(f)
    >
    \TV(f+t\divergence\xi),
    \quad
    \xi \in \Sgn(Df_0-Df),
\end{equation}
where we define for a measure $\mu \in \Meas(\Omega; \R^m)$ the sign
\[
    \Sgn(\mu) \defeq \{ \xi \in L^1(\Omega; \mu) \mid \mu=\xi\abs{\mu} \}.
\]
That is, $-\divergence \xi$ is the total variation flow.

\begin{proof}
    Analogous to Corollary \ref{corollary:tv-l2cost-interior} regarding the $L^2$ cost.
\end{proof}

For $\TGV^2$ we also have an analogous natural condition.

\begin{corollary}[$\TGV^2$ Gaussian denoising with discretised Huber-TV cost]
    \label{corollary:tgv2-hubercost-interior}
    Suppose that the data $f, f_0 \in L^2(\Omega)$ satisfies for some $t, \alpha_2>0$ and
    $\lambda \in V$ the condition
    \begin{equation}
        \label{eq:hubercost-interior-condition-tgv2}
        \TGV^2_{(\alpha_2, 1)}(f)
        >
        \TGV^2_{(\alpha_2, 1)}(f+t\divergence\lambda),
        \quad
        -\divergence \lambda \in \subdiff \costf_{\costhubertv}^V(f).
    \end{equation}
    Then there exists $\bar \epsilon,\bar\gamma > 0$ such any optimal solution $\alpha_{\gamma,\epsilon}=((\alpha_{\gamma,\epsilon})_1, (\alpha_{\gamma,\epsilon})_2)$ to the problem
    \[
        \min_{\alpha \ge 0} \costf_{\costhubertv}^V(f_0-v_\alpha)
    \]
    with
    \[
        \begin{split}
        (v_\alpha, w_\alpha) \in \argmin_{\substack{v \in \BVspace(\Omega)\\ w \in \BDspace(\Omega)}} 
            \Bigl(
            &
            \frac{1}{2}\norm{f-v}_{L^2(\Omega)}^2
            + \alpha_1\huberradon{Dv-w}
            + \alpha_2\huberradon{Ew}
            \\
            &
            + \frac{\epsilon}{2}\norm{(\grad v, \grad w)}_{L^2(\Omega; \R^\DIMdomain \times \R^{\DIMdomain \times \DIMdomain})}^2
            \Bigr)
        \end{split}
    \]
    satisfies $(\alpha_{\gamma,\epsilon})_1, (\alpha_{\gamma,\epsilon})_2 > 0$ whenever $\epsilon \in [0, \bar \epsilon]$, $\gamma \in [\bar\gamma, \infty]$.
\end{corollary}

\begin{proof}
    Analogous to Corollary \ref{corollary:tgv2-l2cost-interior}.
\end{proof}

\section{A few auxiliary results}\label{sec:auxres}

We record in this section some general results that will be useful in the proofs of the main results.
These include the coercivity of the functional $\Jhuber{\freevar}$, recorded in Section \ref{sec:coercivity}.
We then discuss some elementary lower semicontinuity facts in Section \ref{sec:lsc}. 
We provide in Section \ref{sec:strict} some new results for passing from strict convergence to strong convergence



\subsection{Coercivity}
\label{sec:coercivity}

Observe that
\[
    \regf[\bar\alpha]{\mu_1,\ldots,\mu_\NA}
     =  \sup \left\{
        \sum_{j=1}^\NA \bar \alpha_j\left(\mu_j(\psi_j) - \frac{1}{2\gamma}\norm{\psi_j}_{L^2(\Omega; \R^\DIMdomain)}^2\right)
        \middle|
        \begin{array}{l}
            \psi_j \in C_c^\infty(\Omega;  \R^{m_j} \times \R^\DIMdomain), \\
            \sup_{x \in \Omega} \norm{\psi_j(x)}_2^2 \le 1 
        \end{array}
        \right\}.
\]
Thus
\[
    \regforig{\mu} \ge \regf{\mu} \ge \regforig{\mu} - \frac{C'}{2\gamma} \L^n(\Omega)
\]
for some $C'=C'(\bar \alpha)$.
Since $\Omega$ is bounded, it follows that given $\delta>0$, for large enough $\gamma>0$ and every $\epsilon\ge0$ holds
\begin{equation}
    \label{eq:j-comparison}
    \Jorig[\alpha]{u} - \delta \le \Jhuber[\alpha]{u} \le \Jepsilon[\alpha]{u} \le \JepsilonXX[\alpha]{u}{0,\epsilon}.
\end{equation}
We will use these properties frequently. Based on the coercivity and norm equivalence properties in Assumption \ref{ass:a-k} and Assumption \ref{ass:phi}, the following proposition states the important fact that $\JhuberSYM$  is coercive with respect to $\norm{\freevar}_{X}'$ and thus also the standard norm of $X$.

\begin{proposition}
    \label{prop:bv-like-bound}
    Suppose Assumption \ref{ass:a-k} and Assumption \ref{ass:phi} hold,
    and that $\alpha \in \SPACEalphaCompactPosInt$.
    Let $\epsilon \ge 0$ and $\gamma \in (0, \infty]$. Then
    \begin{equation}
        \label{eq:u-bv-like-bound}
        \Jepsilon[\alpha]{u} \to +\infty
        \quad\text{as}\quad
        \norm{u}_{X} \to +\infty.
    \end{equation}
\end{proposition}

\begin{proof}
    Let $\{u^i\}_{i=1}^\infty \subset X$, and suppose $\sup_i \Jepsilon[\alpha]{u^i} < \infty$. Then in particular $\sup_i \Phi(Ku^i) < \infty$. By Assumption \ref{ass:phi} then $\sup_i \norm{Ku^i}_Y < \infty$. But Assumption \ref{ass:a-k} says
    \[
        \norm{u^i}_X' = 
        \sum_{j=1}^{\NA} \radonj{A_j u^i} + \norm{Ku^i}_Y
        \le \Jepsilon[\alpha]{u^i} + \norm{Ku^i}_Y.
    \]
    This implies $\sup_i \norm{u^i}_X' < \infty$.
    By the equivalence of norms in Assumption \ref{ass:a-k}, we immediately obtain \eqref{eq:u-bv-like-bound}.
\end{proof}

\subsection{Lower semicontinuity}
\label{sec:lsc}

We record the following elementary lower semicontinuity facts that we have already used to justify our examples.

\begin{lemma}
    \label{lemma:lsc}
    The following functionals are lower semicontinuous.
    \begin{enumroman}
        \item
            \label{item:lsc-huber}
            $\mu \mapsto \huberradon{\nu-\mu}$ with respect to weak* convergence in $\Meas(\Omega; \R^\DIMkurange)$.
        \item
            \label{item:lsc-ltwo}
            $v \mapsto \norm{f-v}_{\SPACEkuLtwo}^2$ with respect to weak convergence in $\SPACEkuLp{p}$ for any $1<p \le 2$ on a bounded domain $\Omega$.
        \item
            \label{item:lsc-ltwograd}
            $v \mapsto \norm{\grad(f-v)}_{\SPACEgradkuLtwo}^2$ with respect to strong convergence in $\SPACEkuLone$.
    \end{enumroman}
\end{lemma}

\begin{proof}
    In each case, let $\{v^i\}_{i=1}^\infty$ converge to $v$.
    Denoting by $G$ the involved functional, we write it as a
    convex conjugate, $G(v) = \sup\{\iprod{v}{\varphi}-G^*(\varphi)\}$.
    Taking a supremising sequence $\{\varphi^j\}_{j=1}^\infty$ for
    this functional at any point $v$, we easily see lower semicontinuity
    by considering the sequences $\{\iprod{v^i}{\varphi^j}-G^*(\varphi^j)\}_{i=1}^\infty$
    for each $j$.
    In case \ref{item:lsc-ltwo} we use the fact that $\varphi^j \in \SPACEkuLtwo \subset [\SPACEkuLp{p}]^*$ when $\Omega$ is bounded.
    
    In case \ref{item:lsc-huber}, how exactly we write $G(\mu)=\huberradon{\nu-\mu}$
    as a convex conjugate demands explanation. We first of all recall that for $g \in \R^n$, the Huber-regularised norm may be written in dual form as
    \begin{equation}
        \notag
        \huber[\gamma]{g}=
        \sup
        \Bigl\{
            \iprod{q}{g} - \frac{\gamma}{2} \norm{q}_2^2
            \Bigm|
            \norm{q}_2 \le 1
        \Bigr\}.
    \end{equation}
    Therefore, we find that
    \[
        G(\mu)=\sup\left\{
             \mu(\varphi) - \int_\Omega \frac{\gamma}{2} \norm{\varphi(x)}_2^2 \d x
             \middle|
                 \varphi \in C_c^\infty(\Omega),\ 
                 \norm{\varphi(x)}_2 \le \alpha \text{ for every } x \in \Omega
             \right\}.
    \]
    This has the required form.
\end{proof}

We also show here that the marginal regularisation functional $\regfMARGhuber{\freevar}$ is weakly* lower semicontinuous on $Y$. Choosing $K$ as in Example \ref{example:tv-1} and Example \ref{example:tgv2-1}, this provides in particular a proof that $\TV$ and $\TGV^2$ are lower semicontinuous with respect to weak convergence in $L^2(\Omega)$ when $\DIMdomain=1,2$.

\begin{lemma}
    \label{lemma:t-lsc}
    Suppose $\bar\alpha \in \SPACEalphaCompactPosInt$, and Assumption \ref{ass:a-k} holds.
    Then $\regfMARGhuber{\freevar}$ is lower semicontinuous with respect to weak* convergence in $Y$, and continuous with respect to strong convergence in $\range{K}$.
\end{lemma}

\begin{proof}
    Let $v^k \weaktostar v$ weakly* in Y. By the Banach-Steinhaus theorem, the sequence is bounded in $Y$. From the definition
    \[
        \regfMARGhuber{v} \defeq \inf_{\bar v \in \inv Kv} \regf[\bar\alpha]{A\bar v}.
    \]
    Therefore, if we pick $\epsilon>0$ and $\bar v^k \in \inv K v^k$ such that
    \[
        \regf[\bar\alpha]{A\bar v^k} \le \regfMARGhuber[\bar\alpha]{v^k} +\epsilon,
    \]
    then referral to Assumption \ref{ass:a-k}, yields for some constant $c>0$ the bound
    \[
        c \norm{\bar v^k}_X 
        \le \norm{v^k} + \regf[\bar\alpha]{A\bar v^k}
        \le \norm{v^k} + \regfMARGhuber[\bar\alpha]{v^k} + \epsilon.
    \]
    Without loss of generality, we may assume that
    \[
        \liminf_{k \to \infty} \regfMARGhuber[\bar\alpha]{v^k} < \infty,
    \]
    because otherwise there is nothing to prove.
    Then $\{\bar v^k\}_{k=1}^\infty$ is bounded in $X$, and therefore admits a weakly* convergent subsequence. Let $\bar v$ be the limit of this, unrelabelled, sequence. Since $K$ is continuous, we find that $K \bar v = v$. 
    But $\regf[\bar\alpha]{\freevar}$ is clearly weak* lower semicontinuous in $X$; see Lemma \ref{lemma:lsc}.  Thus
    \[
        \regfMARGhuber[\bar\alpha]{v}
        \le
        \regf[\bar\alpha]{A\bar v} 
        \le
        \liminf_{k \to \infty} \regf[\bar\alpha]{A\bar v^k} \le \liminf_{k \to \infty} \regfMARGhuber{v^k} +\epsilon.
    \]
    Since $\epsilon>0$ was arbitrary, this proves weak* lower semicontinuity.

    The see continuity with respect to strong convergence in $\range{K}$, we observe that
    if $v = K u \in \range{K}$, then by the boundedness of the operators $\{A_j\}_{j=1}^\NA$ we get
    \[
        \regfMARGhuber[\bar\alpha]{v} \le \regf[\bar\alpha]{u} \le \regforig[\bar\alpha]{u} \le C \norm{u},
    \]
    for some constant $C>0$.
    So we know that $\regfMARGhuber[\bar\alpha]{\freevar}|\range{K}$ is finite-valued and convex. Therefore it is continuous \cite[Lemma I.2.1]{ekeland1999convex}.
\end{proof}

\subsection{From $\Phi$-strict to strong convergence}
\label{sec:strict}

In Proposition \ref{prop:conditions-onealpha}, forming part of the proof of our main theorems, we will need to pass from ``$\Phi$-strict convergence'' of $Ku^k$ to $v$ to strong convergence, using the following lemmas. The former means that $\Phi(Ku^k) \to \Phi(v)$ and $Ku^k \weakto v$ weakly* in $Y$. By strong convexity in a Banach space $Y$, we mean the existence of $\gamma>0$ such that for every $y \in Y$ and $z \in \subdiff \Phi(y) \subset Y^*$ holds
\[
    \Phi(y')-\Phi(y) \ge \dprod{z}{y'-y} + \frac{\gamma}{2}\norm{y'-y}_Y^2,
    \quad (y' \in Y),
\]
where $\dprod{z}{y}$ denotes the dual product, and the subdifferential $\subdiff \Phi(y)$ is defined by $z$ satisfying the same expression with $\gamma=0$.
With regard to more advanced strict convergence results, we point the reader to \cite{delladio1991lower,rindler2013strictly,kristensen2010relaxation}.

\begin{lemma}
    \label{lemma:strict-to-strong}
    Suppose $Y$ is a Banach space, and $\Phi: Y \to (-\infty, \infty]$ strongly convex. If $v^k \weakto \hat v \in \Dom \subdiff \Phi$ weakly* in $Y$ and $\Phi(v^k) \to \Phi(\hat v)$, then $v^k \to \hat v$ strongly in $Y$.
\end{lemma}

\begin{remark}
By standard convex analysis \cite{ekeland1999convex}, $v \in \Dom \subdiff \Phi$ if $\Phi$ has a finite-valued point of continuity and $v \in \interior \Dom \Phi$.
\end{remark}

\begin{proof}
    We first of all note that $-\infty < \Phi(\hat v) < \infty$ because $v \in \Dom \subdiff\Phi$ implies $v \in \Dom \Phi$.
    Let us pick $z \in \subdiff \Phi(\hat v)$. From the strong convexity of $\Phi$, for some $\gamma>0$ then
    \[
        \Phi(v^k) - \Phi(\hat v) \ge \dprod{z}{v^k-\hat v} + \frac{\gamma}{2}\norm{v^k-\hat v}_Y^2.
    \]
    Taking the limit infimum, we observe
    \[
         0
         =\Phi(\hat v) - \Phi(\hat v) d
         \ge \liminf_{k \to \infty} \frac{\gamma}{2}\norm{v^k-\hat v}_Y^2.
    \]
    This proves strong convergence.
\end{proof}

We now use the lemma to show strong convergence of minimising sequences.

\begin{lemma}
    \label{lemma:phi-to-zero}
    Suppose $\Phi$ is strongly convex, satisfies Assumption \ref{ass:phi}, and that $C \subset Y$ is non-empty, closed, and convex with $\interior C \isect \Dom \Phi \ne \emptyset$.
    Let 
    \[
        \hat v \defeq \argmin_{v \in C} \Phi(v).
    \]
    If $\{v^k\}_{k=1}^\infty \subset Y$ with $\lim_{k \to \infty} \Phi(v^k) = \Phi(\hat v)$, then $v^k \to \hat v$ strongly in $Y$.
\end{lemma}

\begin{proof}
    By the strict convexity of $\Phi$, implied by strong convexity, and the assumptions on $C$, $\hat v$ is unique and well-defined.
    Moreover $\hat v \in \Dom \subdiff \Phi$. Indeed, our assumptions show the existence of a point $v \in \interior C \isect \Dom \Phi$. The indicator function $\delta_C$ is then continuous at $v$, and so standard subdifferential calculus (see, e.g., \cite[Proposition I.5.6]{ekeland1999convex}) implies that $\subdiff(\Phi+\delta_C)(\hat v)=\subdiff \Phi(\hat v)+\subdiff \delta_C(\hat v)$. But $0 \in \subdiff(\Phi+\delta_C)(\hat v)$ because $\hat v \in \argmin_{v \in Y} \Phi(v)+\delta_C(v)$. This implies that $\subdiff \Phi(\hat v)\ne\emptyset$. Consequently also $\hat v \in \Dom \Phi$, and $\Phi(\hat v) \in \R$

    Using the coercivity of $\Phi$ in Assumption \ref{ass:phi} we then find that $\{v^k\}_{k=1}^\infty$ is bounded in $Y$, at least after moving to an unrelabelled tail of the sequence with $\Phi(v^k) \le \Phi(\hat v)+1$. 
    Since $Y$ is a dual space, the unit ball is weak* compact, and we deduce the existence of a subsequence, unrelabelled, and $v \in Y$ such that $v^k \weakto v$ weakly* in $Y$. 
    By the weak* lower semicontinuity (Assumption \ref{ass:phi}), we deduce
    \[
        \Phi(v) \le \liminf_{k \to \infty} \Phi(v^k) = \Phi(\hat v).
    \]
    Since each $v^k \in C$, and $C$ is closed, also $v \in C$.
    Therefore, by the strict convexity and the definition of $\hat v$, necessarily $v=\hat v$.
    Therefore $v^k \weakto \hat v$ weakly* in $Y$, and $\Phi(v^k) \to \Phi(\hat v)$.
    Lemma \ref{lemma:strict-to-strong} now shows that $v^k \to \hat v$ strongly in $Y$. 
\end{proof}

\section{Proofs of the main results}
\label{sec:proof}

We now prove the existence, continuity, and non-degeneracy (interior solution) results of Section \ref{sec:main} through a series of lemmas and propositions, starting from general ones that are then specialised to provide the natural conditions presented in Section \ref{sec:main}.

\subsection{Existence and lower semicontinuity under lower bounds}

Our principal tool for proving existence is the following proposition. We will in the rest of this section concentrate on proving the existence of the set $\mathcal{K}$ in the statement. We base this on the natural conditions of Section \ref{sec:main}.

\begin{proposition}[Existence on compact parameter domain]
    \label{prop:existence-compact}
    Suppose Assumption \ref{ass:a-k} and \ref{ass:phi} hold. With $\epsilon \ge 0$ and $\gamma \in (0, \infty]$ fixed, if there exists a compact set $\mathcal{K} \subset \SPACEalphaCompactPosInt$ with
    \begin{equation}
        \label{eq:interior-solution-compact}
        \inf_{\alpha \in \SPACEalphaCompactPos \setminus \mathcal{K}}
            \costf(\costK u_{\alpha,\gamma,\epsilon}) 
        >
        \inf_{\alpha \in \SPACEalphaCompactPos }
            \costf(\costK u_{\alpha,\gamma,\epsilon}),
    \end{equation}
    then there exists a solution $\alpha_{\gamma,\epsilon} \in \SPACEalphaCompactPosInt$ to \eqref{eq:learn-numerical-single}. Moreover, the mapping
    \[
        \ValMap{\epsilon}(\alpha) \defeq \costf(\costK u_{\alpha,\gamma,\epsilon}),
    \]
    is lower semicontinuous within $\SPACEalphaCompactPosInt$.
\end{proposition}


The proof depends on the following two lemmas that will be useful later on as well.

\begin{lemma}[Lower semicontinuity of the fidelity with varying parameters]
    \label{lemma:jhuber-lsc}
    Suppose Assumption \ref{ass:a-k}, \ref{ass:phi}, and \ref{ass:sol-smoothness} hold.
    Let $u^k \weaktostar u$ weakly* in $X$, and $(\alpha^k, \gamma^k, \epsilon^k) \to (\alpha, \gamma, \epsilon) \in \SPACEalphaCompactPosInt \times (0, \infty] \times (0, \infty)$. Then
    \[
        \JepsilonXX[\alpha]{u}{\gamma, \epsilon}
        \le
        \liminf_{k \to \infty} \JepsilonXX[\alpha^k]{u^k}{\gamma^k, \epsilon^k}.
    \]
\end{lemma}

\def\pointPhiKfinite{\invK f}

\begin{proof}
    Let $\hat \epsilon > 0$ be such that $\alpha_j \ge \hat \epsilon$, ($j=1,\ldots,\NA$).
    We then deduce for large $k$ and some $C=C(\Omega, \gamma)$ that
    \begin{equation}
        \label{eq:jepsilon-hat-epsilon-bound}
        \begin{split}
        \limsup_k \JepsilonXX[\hat \epsilon, \ldots, \hat \epsilon] {u^k}{\gamma,\epsilon^k}
        &
        \le
        \limsup_k \JepsilonXX[\hat \epsilon, \ldots, \hat \epsilon] {u^k}{\gamma^k,\epsilon^k}
        +C
        \\
        &
        \le
        \limsup_k \JepsilonXX[\alpha^k] {u^k}{\gamma^k,\epsilon^k} + C
         < \infty.
        \end{split}
    \end{equation}
    Here we have assumed the final inequality to hold. This comes without loss of generality, because otherwise there is nothing to prove.
    Observe that this holds even if $\{\alpha^k\}_{k=1}^\infty$ is not bounded.
    In particular, if $\epsilon >0$, restricting $k$ to be large, we may assume that
    \begin{equation}
        \label{eq:hilbert-bound}
        C_1 \defeq \Smoother(u^k) < \infty.
    \end{equation}

    We recall that
    \begin{equation}
        \label{eq:jepsilon-recall}
        \Jepsilon[\alpha]{u} :=
            \epsilon\Smoother(u)
            +
            \Phi(Ku)
            +
            \sum_{j=1}^{\NA} \alpha_j \huberj{A_j u}.
    \end{equation}
    We want to show lower semicontinuity of each of the terms in turn.
    We start with the smoothing term.
    If $\epsilon>0$, using \eqref{eq:hilbert-bound}, we write
    \[
        \epsilon^k\Smoother(u^k)
        =
        (\epsilon^k-\epsilon)\Smoother(u^k)
        +
        \epsilon\Smoother(u^k)
        \le
        (\epsilon^k-\epsilon)\frac{C_1}{2}
        +
        \epsilon\Smoother(u^k).
    \]
    By the convergence $\epsilon^k \to \epsilon$, and the weak* lower semicontinuity of $\Smoother$, we find that
    \begin{equation}
        \label{eq:hilbert-lsc}
        \epsilon \Smoother(u)
        \le
        \liminf_{k \to \infty}
        \epsilon^k \Smoother(u^k)
    \end{equation}
    If $\epsilon=0$, we have
    \[
        \epsilon\Smoother(u)=0 \cdot \infty = 0,
    \]
    while still
    \[
        0 \le \sup_k \epsilon^k \Smoother(u^k) < \infty.
    \]
    Thus \eqref{eq:hilbert-lsc} follows.
    
    The fidelity term $\Phi \circ K$ is weak* lower semicontinuous by the continuity of $K$ and the weak* lower semicontuity of $\Phi$.
    It therefore remains to consider the terms in \eqref{eq:jepsilon-recall} involving both the regularisation parameters $\alpha$, as well as the Huberisation parameter $\gamma$. 
    Indeed using the dual formulation \eqref{eq:huber-twonorm} of the Huberised norm, we have
    for some constant $C'=C'(\hat \epsilon, \Omega)$ that
    \[
        \begin{split}
        \huberjX{A_j u^k}{\gamma^k}
        &
        =
        \sup_{\norm{\varphi(x)} \le 1} \int \varphi \d A_j u^k - \frac{1}{2\gamma^k} \norm{\varphi}^2 \d x
        \\
        &
        \ge
        \sup_{\norm{\varphi(x)} \le 1} \int_\Omega \varphi \d A_j u^k + \frac{1}{2\gamma} \norm{\varphi}^2 \d x
        - C' \abs{\inv\gamma-\inv{(\gamma^k)}}
        \\
        &
        =
        \huberjX{A_j u^k}{\gamma}
        - C' \abs{\inv\gamma-\inv{(\gamma^k)}}.
        \end{split}
    \]
    Thus, if $\alpha \in \SPACEalphaPos$, we get
    \[
        \begin{split}
        \alpha^k_j \huberjX{A_j u^k}{\gamma^k}
        &
        =
        \alpha_j \huberjX{A_j u^k}{\gamma^k}
        +
        (\alpha^k_j-\alpha_j) \huberjX{A_j u^k}{\gamma^k}
        \\
        &
        \ge
        \alpha_j\huberjX{A_j u^k}{\gamma^k}
        -
        \abs{\alpha^k_j-\alpha_j} \huberjX{A_j u^k}{\gamma^k}
        \\
        &
        \ge
        \alpha_j \huberjX{A_j u^k}{\gamma}
        - C' \abs{\inv\gamma-\inv{(\gamma^k)}}
        - \abs{\alpha^k_j-\alpha_j} \huberjX{A_j u^k}{\gamma^k}.
        \end{split}
    \]
    It follows from \eqref{eq:jepsilon-hat-epsilon-bound} that the sequence $\{\huberjX{A_j u^k}{\gamma^k}\}_{k=1}^\infty$ is bounded in $\SPACEj$ for each $j=1,\ldots,\NA$. Thus
    \[
        \liminf_{k \to \infty} \alpha^k_j \huberjX{A_j u^k}{\gamma^k}
        \ge
        \liminf_{k \to \infty} \alpha_j \huberjX{A_j u^k}{\gamma}
        \ge
        \alpha_j \huberjX{A_j u}{\gamma},
    \]
    where the final step follows from Lemma \ref{lemma:lsc}.
    
    It remains to consider the case that $\alpha \in \SPACEalphaCompactPos \setminus \SPACEalphaPos$, i.e., when $\alpha_\ell=\infty$ for some $\ell$. We may pick sequences $\{\beta_j^\ell\}_{\ell=1}^\infty$, ($j=1,\ldots,\NA$), such that $\beta_j^\ell \upto \alpha_j$. Further, we may find $\{\beta_j^{k,\ell}\}_{\ell=1}^\infty $ such that $\beta_j^{k,\ell} \le \alpha_j^k$ with $\alpha_j^k = \lim_{\ell \to \infty} \beta_j^{k,\ell}$ and $\beta_j^\ell = \lim_{k \to \infty} \beta_j^{k,\ell}$. Then, by the bounded case studied above
    \[
        \liminf_{k \to \infty} \alpha^k_j\huberjX{A_j u^k}{\gamma^k}
        \ge
        \liminf_{k \to \infty} \beta_j^{k, \ell}\huberjX{A_j u^k}{\gamma^k}
        \ge
        \beta_j^\ell \huberjX{A_j u}{\gamma}.
    \]
    But $\{\alpha^k_j\huberjX{A_j u^k}{\gamma^k}\}_{k=1}^\infty$ is bounded
    by \eqref{eq:jepsilon-hat-epsilon-bound}, and
    \[
        \liminf_{\ell \to \infty} 
        \beta_j^\ell\huberjX{A_j u}{\gamma}
        \ge
        \alpha_j\huberjX{A_j u}{\gamma}.
    \]
    Thus lower semicontinuity follows.
\end{proof}

\begin{lemma}[Convergence of reconstructions away from boundary]
    \label{lemma:strong-approximation}
    Suppose Assumption \ref{ass:a-k}, \ref{ass:phi}, and \ref{ass:sol-smoothness} hold.
    Let $(\alpha^k, \gamma^k, \epsilon^k) \to (\alpha, \gamma, \epsilon)$ in $\SPACEalphaCompactPosInt \times (0, \infty] \times [0, \infty)$. Then we can find
    $u_{\alpha,\gamma,\epsilon} \in \argmin \Jepsilon[\alpha]{\freevar}$ and extract a subsequence satisfying
    \begin{subequations}
    \begin{align}
        \label{eq:strong-approximation-j}
        \JepsilonXX[\alpha^k]{u_{\alpha^k,\gamma^k,\epsilon^k}}{\gamma^k,\epsilon^k}
        &
        \to
        \JepsilonXX[\alpha]{u_{\alpha,\gamma,\epsilon}}{\gamma,\epsilon},
        \\
        \label{eq:strong-approximation-weakstar}
        u_{\alpha^k,\gamma^k,\epsilon^k} &
        \weaktostar
        u_{\alpha,\gamma,\epsilon}
        \quad
        \text{weakly* } X, \quad\text{and}
        \\
        \label{eq:strong-approximation-kstrong}
        Ku_{\alpha^k,\gamma^k,\epsilon^k}
        &
        \to
        Ku_{\alpha,\gamma,\epsilon}
        \quad
        \text{strongly in } Y.
    \end{align}
    \end{subequations}
\end{lemma}

\begin{proof}
    By Lemma \ref{lemma:jhuber-lsc}, we have
    \[
       \liminf_{k \to \infty} 
        \JepsilonXX[\alpha^k]{u_{\alpha^k,\gamma^k,\epsilon^k}}{\gamma^k,\epsilon^k}
        \ge
        \JepsilonXX[\alpha]{u_{\alpha,\gamma,\epsilon}}{\gamma,\epsilon}
        =
        \min_{u \in X} \JepsilonXX[\alpha]{u}{\gamma,\epsilon}.
    \]
    We also want the opposite inequality
    \begin{equation}
        \label{eq:strong-approximation-j-upper}
        \limsup_{k \to \infty} 
        \JepsilonXX[\alpha^k]{u_{\alpha^k,\gamma^k,\epsilon^k}}{\gamma^k,\epsilon^k}
        \le
        \JepsilonXX[\alpha]{u_{\alpha,\gamma,\epsilon}}{\gamma,\epsilon}.
    \end{equation}
    Let $\delta>0$. If $\epsilon=0$, we use Assumption \ref{ass:sol-smoothness} on $u=u_{\alpha,\gamma,\epsilon}$, to produce $u^\delta$. Otherwise, we set $u^\delta=u_{\alpha,\gamma,\epsilon}$. In either case
    \[
        \JepsilonXX[\alpha]{u^\delta}{\gamma, \epsilon} \le \JepsilonXX[\alpha]{u_{\alpha,\gamma,\epsilon}}{\gamma, \epsilon} + \delta.
    \]
    In particular $A_j u^\delta=0$ if $\alpha_j=\infty$.
    Then for large enough $k$ we obtain
    \begin{equation}
        \label{eq:u-conv-0}
        \begin{split}
        \JepsilonXX[\alpha^k]{u_{\alpha^k,\gamma^k,\epsilon^k}}{\gamma^k,\epsilon^k}
        &
        \le
        \JepsilonXX[\alpha^k]{u^\delta}{\gamma^k,\epsilon^k}
        \\
        &
        \le
        \JepsilonXX[\alpha]{u^\delta}{\gamma,\epsilon} + \delta
        \\
        &
        \le
        \JepsilonXX[\alpha]{u_{\alpha,\gamma,\epsilon}}{\gamma, \epsilon} + 2\delta.
        \end{split}
    \end{equation}
    Since $\delta>0$ was arbitrary, this proves \eqref{eq:strong-approximation-j-upper} and consequently \eqref{eq:strong-approximation-j}, that is
    \begin{equation}
        \label{eq:epsilon-to-zero-conv}
        \lim_{k \to \infty} 
        \JepsilonXX[\alpha^k]{u_{\alpha^k,\gamma^k,\epsilon^k}}{\gamma^k,\epsilon^k}
        =
        \min_{u \in X} \Jepsilon[\alpha]{u}=\Jorig[\alpha]{u_{\alpha,\gamma,\epsilon}}
        \le \Jepsilon[\alpha]{0}=\epsilon\Smoother(0)+\Phi(0) < \infty.
    \end{equation}
    Minding Proposition \ref{prop:bv-like-bound}, this allows us to extract a subsequence of $\{u_{\alpha,\gamma^k\epsilon^k}\}_{k = 1}^\infty$, unrelabelled, and convergent weakly* in $X$ to some
    \[
        \tilde u \in \argmin_{u \in X} \Jorig[\alpha]{u}.
    \]
    We may choose $u_{\alpha,\gamma,\epsilon} \defeq \tilde u$. This shows \eqref{eq:strong-approximation-weakstar}.

    If $K$ is compact, we may further assume that $Ku_{\alpha^k, \gamma^k, \epsilon^k} \to Ku_{\alpha,\gamma,\epsilon}$ strongly in $Y$, showing  \eqref{eq:strong-approximation-kstrong}.
    If $K$ is not compact, $\Phi$ is continuous and strongly convex by Assumption \ref{ass:phi}, and we still have $\Phi(Ku_{\alpha^k, \gamma^k, \epsilon^k}) \to \Phi(Ku_{\alpha,\gamma,\epsilon})$. The assumptions of Lemma \ref{lemma:strict-to-strong} are therefore satisfied. This shows \eqref{eq:strong-approximation-kstrong}.
\end{proof}

From the lower semicontinuity Lemma \ref{lemma:jhuber-lsc}, we immediately obtain the following standard result.

\begin{theorem}[Existence of solutions to the reconstruction sub-problem]
    \label{thm:subproblem-existence}
    Let $\Omega \subset \R^\DIMdomain$ be a bounded open domain. Suppose Assumption \ref{ass:a-k} and \ref{ass:phi} hold,  and that $\alpha \in \SPACEalphaCompactPosInt$, $\epsilon \ge 0$, and $\gamma \in (0, \infty]$.
    Then \eqref{eq:denoise-numerical-single} admits a minimiser $u_{\alpha,\gamma,\epsilon} \in X \isect \Dom \epsilon\Smoother$.
\end{theorem}


\begin{proof}
    By Lemma \ref{lemma:jhuber-lsc}, fixing $(\alpha^k,\gamma^k,\epsilon^k)=(\alpha,\gamma,\epsilon)$, the functional $\Jepsilon[\alpha]{\freevar}$ is lower semicontinuous with respect to weak* convergence in $X$. So we just have to establish a weak* convergent minimising sequence. Towards this end, we let $\{u^k\}_{k=1}^\infty \subset \Dom \epsilon\Smoother$ be a minimising sequence for \eqref{eq:denoise-numerical-single}.
    We may assume without loss of generality that
    $\sup_k \Jepsilon[\alpha]{u^k} < \infty$.
    By Proposition \ref{prop:bv-like-bound} and the inequality $\JhuberSYM \le \JepsilonSYM$, we deduce $\sup_k \norm{u^k}_X < \infty$.
    After possibly switching to a subsequence, unrelabelled, we may therefore assume $\{u^k\}_{k=1}^\infty$ weakly* convergent in $X$ to some $\hat u \in X$.
    this proves the claim.
\end{proof}

\begin{proof}[Proof of Proposition \ref{prop:existence-compact}] 
    Let us take a sequence $\{\alpha^k\}_{k=1}^\infty$ convergent to $\alpha \in \SPACEalphaCompactPosInt$.
    Application of Lemma \ref{lemma:strong-approximation} with $\gamma^k=\gamma$ and $\epsilon^k=\varepsilon$ and the weak* lower semicontinuity of $\costf$ immediately show the lower semicontinuity of $\ValMap{\epsilon}$ within $\SPACEalphaCompactPosInt$.

    Finally, if $\{\alpha^k\}_{k=1}^\infty$ is a minimising sequence for \eqref{eq:learn-numerical-single}, by assumption we may take it to lie in $\mathcal{K}$. By the compactness of $\mathcal{K}$,  we may assume the sequence convergent to some $\alpha \in \mathcal{K}$.
    By the lower semicontinuity established above, $\hat u=u_{\alpha,\gamma,\epsilon}$ is a solution to \eqref{eq:learn-numerical-single}.
\end{proof}

\subsection{Towards $\Gamma$-convergence and continuity of the solution map}

The next lemma, immediate from the previous one, will form the first part of the proof of continuity of the solution map. As its condition, we introduce a stronger form of \eqref{eq:interior-solution-compact} that is uniform over a range of $\epsilon$ and $\gamma$.

\begin{lemma}[$\Gamma$-lower limit of the cost map in terms of regularisation]
    \label{lemma:gamma-lower}
    Suppose Assumption \ref{ass:a-k}, \ref{ass:phi}, and \ref{ass:sol-smoothness} hold.
    Let $\mathcal{K} \subset \SPACEalphaCompactPosInt$ be compact.
    Then
    \begin{equation}
        \label{eq:gamma-lower}
        \ValM{\gamma,\epsilon}(\alpha)
        \le \liminf_{(\alpha', \gamma',\epsilon') \to (\alpha, \gamma,\epsilon)}
        \ValM{\gamma',\epsilon'}(\alpha').
    \end{equation}
    when the convergence is within $\mathcal{K} \times [\bar \gamma, \infty] \times [0, \bar\epsilon]$.
\end{lemma}

\begin{proof}
    Consequence of \eqref{eq:strong-approximation-weakstar} of Lemma \eqref{lemma:strong-approximation} and the weak* lower semicontinuity of $F$.
\end{proof}

The next lemma will be used to get partial strong convergence of minimisers as we approach $\BD \SPACEalphaPos$. This will then be used to derive simplified conditions for this not happening. This result is the counterpart of Lemma \ref{lemma:strong-approximation} that studied convergence of reconstructions away from the boundary, and depends on the additional Assumption \ref{ass:aj-zeroterm-approx}. This is the only place where we use the assumption, and replacing this lemma by one with different assumptions would allow us to remove Assumption \ref{ass:aj-zeroterm-approx}.

\begin{lemma}[Convergence of reconstructions at the boundary]
    \label{lemma:alpha-to-zero}
    Suppose Assumption \ref{ass:a-k}, \ref{ass:phi}, and \ref{ass:aj-zeroterm-approx} hold, and that $\Phi$ is strongly convex.
    Suppose $\{(\alpha^k,\gamma^k,\epsilon^k)\}_{k=1}^\infty \in \SPACEalphaCompactPosInt \times (0, \infty] \times [0, \bar\epsilon]$ satisfies $\alpha^k \to \alpha \in \BD \SPACEalphaCompactPos$.
    If $\bar \epsilon=0$ or Assumption \ref{ass:sol-smoothness} holds and $\bar \epsilon>0$ is small enough, then $Ku_{\alpha^k,\gamma^k,\epsilon^k} \to f$ strongly in $Y$.
\end{lemma}

\begin{proof}
    We denote for short $u^k \defeq u_{\alpha^k,\gamma^k,\epsilon^k}$, and note that $f$ is unique by the strong convexity of $\Phi$.
    Since $\bar\alpha \in \BD \SPACEalphaCompactPos$, there exist an index $\ell \in \{1,\ldots,\NA\}$ such that $\alpha^k_\ell \to 0$. We let $\ell$ be the first such index, and pick arbitrary $\delta > 0$.
    We take $\bar f_{\delta,\ell}$ as given by Assumption \ref{ass:aj-zeroterm-approx},
    observing that the construction still holds with Huberisation, that is, for any $\gamma \in (0, \infty]$ and in particular any $\gamma=\gamma^k$, we have
    \begin{align}
        \notag
        \Phi(K\bar f_{\delta,\ell})
        &
        < \delta + \Phi(f),
        \\
        \notag
        \huberj[\ell]{A_\ell \bar f_{\delta,\ell}}
        &
        \le
        \radonj[\ell]{A_\ell \bar f_{\delta,\ell}}
        < \infty,
        \quad\text{and},
        \\
        \notag
        \sum_{j \ne \ell} \huberj{A_j \bar f_{\delta,\ell}}
        &
        = 0.
    \end{align}
    If we are aiming for $\bar\epsilon>0$, let us also pick $\tilde f_{\delta,\ell}$ by application of Assumption \ref{ass:sol-smoothness} to $u=\bar f_{\delta,\ell}$.
    Otherwise, with $\bar\epsilon=0$, let us just set $\tilde f_{\delta,\ell}=\bar f_{\delta,\ell}$.
    Since
    \[
        u^k \in \argmin_{u \in X} \JepsilonXX[\alpha^k]{u}{\gamma^k,\epsilon^k},
    \]
    we have
    \[
        \begin{split}
        \Phi(K u^k)
        &
        \le
        \JepsilonXX[\alpha^k]{u^k}{\gamma^k,\epsilon^k}
        \le
        \JepsilonXX[\alpha^k]{\tilde f_{\delta,\ell}}{\gamma^k,\epsilon^k}
        \le
        \JepsilonXX[\alpha^k]{\bar f_{\delta,\ell}}{\gamma^k,\epsilon^k} + \delta
        \\
        &
        =
        \epsilon^k\Smoother(\bar f_{\delta,\ell})
        +
        \Phi(K \bar f_{\delta,\ell})
        +\regfX[\alpha^k]{A \bar f_{\delta,\ell}}{\gamma^k}
        +\delta
        \\
        &
        \le
        \epsilon^k\Smoother(\bar f_{\delta,\ell})
        +
        \Phi(f)
        + \alpha_\ell^k \radonj[\ell]{A_\ell \bar f_{\delta,\ell}}
        + 2\delta.
        \end{split}
    \]
    Observe that it is no problem if some index $\alpha^k_j=\infty$, because by definition as a minimiser $u^k$ achieves smaller value than $\bar f_{\delta,\ell}$ above, and for the latter $\huberjX{A_j\bar f_{\delta,\ell}}{\gamma^k}=0$.
    Choosing $\bar \epsilon>0$ small enough, it follows for $\epsilon \in [0, \bar \epsilon]$ that
    \[
        0
        \le
        \Phi(K u^k) - \Phi(f)
        \le
        2\delta + \alpha_\ell^k \radonj[\ell]{A_\ell \bar f_{\delta,\ell}}.
    \]
    Choosing $k$ large enough, we thus see that
    \[
        0
        \le
        \Phi(K u^k) - \Phi(f)
        \le
        (2+\alpha_\ell)\delta.
    \]
    Letting $\delta \downto 0$, we see that $\Phi(Ku^k) \to \Phi(f)$. 
    Lemma \ref{lemma:phi-to-zero} with $C=Y$ therefore shows that $Ku^k \to f$ strongly in $Y$.
\end{proof}

\subsection{Minimality and co-coercivity}

Our remaining task is to show the existence of $\mathcal{K}$ for \eqref{eq:interior-solution-compact}, and of a uniform $\mathcal{K}$ -- see \eqref{eq:interior-solution-compact-strong} below -- for the application of Lemma \ref{lemma:gamma-lower}. When the fidelity and cost functionals satisfy some additional conditions, we will now reduce this to the existence of $\tilde\alpha \in \SPACEalphaPosInt$ satisfying $\costf(u_{\tilde\alpha}) < \costf(\bar f)$ for a specific $\bar f \in \inv K f$. 
So far, we have made no reference to the data, the ground-truth $f_0$ or the corrupted measurement data $f$. We now assume this in an abstract way, and need a type of source condition, called minimality, relating the ground truth $f_0$ to the noisy data $f$. We will get back to how this is obtained later.

\begin{definition}
    Let $p>0$.
    We say that $\bar v \in X$ is \emph{$(K,p)$-minimal} if there exists $C \ge 0$ and $\varphi_{\bar v} \in Y^*$ such that
    \[
        F(u)-F(\bar v)
        \ge
        \dprod{\varphi_{\bar v}}{K(u-\bar v)}
        -\frac{C}{p}\norm{K(u-\bar v)}_Y^p.
    \]
\end{definition}

\begin{remark}
    If we can take $C=0$, then the final condition just says that $K^* \varphi_{\bar v} \in \subdiff F(\bar v)$.
    This is a rather strong property.
    Also, instead of $t \mapsto t^p$, we could in the following proofs use any strictly increasing energy $\psi: [0, \infty) \to [0, \infty)$, $\psi(0)=0$.
\end{remark}


To deal with the smoothing term $\epsilon\Smoother$ with $\epsilon>0$, we also need co-coercivity; for the justification of the term for the condition in \eqref{eq:co-coercive} below, more often seen in the context of monotone operators, we refer to the equivalences in \cite[Theorem 18.15]{bauschke2011convex}.

\begin{definition}
    We say that $\costf$ is \emph{$(K, p)$-co-coercive at $(u^*, \lambda^*) \in X \times X^*$}, $\lambda^* \in \subdiff F(u^*)$, if
    \begin{equation}
        \label{eq:co-coercive}
        \costf(u)-\costf(u^*) \le \dprod{\lambda^*}{u-u^*} 
        + \frac{C}{p}\norm{K(u-u^*)}_Y^p,
        \quad
        (u \in X).
    \end{equation}
    If $\costf$ is $(K,p)$-co-coercive at $(u, \lambda)$ for every $u \in X$ and $\lambda \in \subdiff F(u)$, we say that $\costf$ is $(K,p)$-co-coercive.
    If $p=2$, we say that $\costf$ is simply $K$-co-coercive.
\end{definition}

\begin{remark}
    In essence, $K$-co-coercivity requires $F=F_0 \circ K$ and usual ($I$-)co-coercivity of $F_0$.
\end{remark}

\begin{lemma}
    \label{lemma:k-convergence}
    Suppose $\bar v \in X$ is $(K, p)$-minimal.
    If $\{u^k\}_{k=1}^\infty \subset X$ satisfies $Ku^k \to Kv$ in $Y$, then
    \begin{equation}
        \label{eq:k-convergence}
        \costf(\bar v) \le \liminf_{k \to \infty} \costf(u^k).
    \end{equation}
    If, moreover, $\costf$ is $(K, p)$-co-coercive at $(\bar v, K^* \varphi_{\bar v})$, then
    \begin{equation}
        \label{eq:f-k-continuity}
        \costf(\bar v) = \lim_{k \to \infty} \costf(u^k).
    \end{equation}
\end{lemma}
\begin{proof}
    For \eqref{eq:k-convergence}, we use the $(K, p)$-minimality of $\bar v$ to obtain
    \[
        \costf(u^k) - \costf(\bar v) \ge \dprod{\varphi_{\bar v}}{Ku^k-v}
            - \frac{C}{p}\norm{Ku^k-v}_Y^p,
        \quad 
        (k=1,2,3,\ldots).
    \]
    Taking the limit, it follows that
    \[
        \liminf_{k \to \infty} \costf(u^k)
        \ge
        \costf(\bar v).
    \]

    If we additionally have the $(K, p)$-co-coercivity at $(\bar v, K^* \varphi_{\bar v})$, then, likewise
    \[
        \costf(u^k) - \costf(\bar v) \le \dprod{\varphi_{\bar v}}{Ku^k-v}
                + \frac{C}{2}\norm{Ku^k-v}_{Y}^p,
        \quad 
        (k=1,2,3,\ldots).
    \]
    From this we immediately get
    \[
        \limsup_{k \to \infty} \costf(u^k)
        \le
        \costf(\bar v).
        \qedhere
    \]
\end{proof}


\begin{proposition}[One-point conditions under co-coercivity]
    \label{prop:conditions-onealpha}
    Suppose Assumption \ref{ass:a-k}, \ref{ass:phi} and \ref{ass:aj-zeroterm-approx} hold, and that $\Phi$ is \emph{strongly} convex. If $\bar f$ is $(K, q)$-minimal and
     \begin{equation}
        \label{eq:costf-tildeu}
        \text{some}\quad
        \tilde\alpha \in \SPACEalphaPosInt
        \quad\text{and}\quad
        u_{\tilde\alpha} \in \argmin_{u \in X} \Jorig[\tilde\alpha]{u}
        \quad\text{satisfy}\quad
        \begin{cases}
            \costf(u_{\tilde\alpha}) < \costf(\bar f), \quad \text{and}\\
            u_\alpha \text{ is } (K, p)\text{-minimal},
        \end{cases}
    \end{equation}
    then there exist $\bar\gamma,\bar\epsilon>0$ such that the following hold.
    \begin{enumroman}
        \item
        \label{item:a-exist-zero}
        For each $\epsilon \in [0, \hat \epsilon]$ and $\gamma \in [\bar\gamma,\infty]$ there exists a compact set $\mathcal{K} \subset \SPACEalphaCompactPosInt$ such that \eqref{eq:interior-solution-compact} holds.
        \item
        \label{item:a-exist-barepsilon}
        If, moreover, Assumption \ref{ass:sol-smoothness} holds and $\costf$ is $(K,p)$-co-coercive for any $p>0$, then there exist a compact set $\mathcal{K} \subset \SPACEalphaCompactPosInt$ such that
        \begin{equation}
            \label{eq:interior-solution-compact-strong}
            \inf_{\alpha \in \SPACEalphaCompactPos \setminus \mathcal{K}}
                \costf(\costK u_{\alpha,\gamma,\epsilon}) 
            >
            \inf_{\alpha \in \SPACEalphaCompactPos }
                \costf(\costK u_{\alpha,\gamma,\epsilon}),
            \quad
            (\gamma \in [\bar\gamma, \infty], \epsilon \in [0, \bar \epsilon]).
        \end{equation}

    \end{enumroman}
\end{proposition}

In both cases, the existence of $\mathcal{K}$ says that every solution $\bar \alpha$ to \eqref{eq:learn-numerical-single} satisfies $\bar \alpha \in \mathcal{K}$.


\begin{proof}
    We note that $f$ is unique by the strong convexity of $\Phi$.
    Let us first prove \ref{item:a-exist-zero}.
    In fact, let us pick $\doublebar\gamma,\doublebar\epsilon>0$ and assume with $\tilde \alpha$ fixed that
    \begin{equation}
        \label{eq:costf-tildeu-epsilon}
        u_{\tilde\alpha,\gamma,\epsilon} \in \argmin_{u \in X} \Jepsilon[\tilde\alpha]{u}
        \quad\text{satisfy}\quad
        \costf(u_{\tilde\alpha,\gamma,\epsilon}) < \costf(\bar f),
        \quad (\gamma \in [\doublebar\gamma,\infty], \epsilon \in [0, \doublebar\epsilon]).
    \end{equation}
    We want to show the existence of a compact set $\mathcal{K} \subset \SPACEalphaCompactPos$ such that solutions $\hat \alpha$ to \eqref{eq:learn-numerical-single} satisfy $\hat \alpha \in \mathcal{K}$ whenever $(\gamma, \epsilon) \in [\bar\gamma,\infty] \times [0, \bar \epsilon]$ for $\bar\gamma \in [\doublebar\gamma,\infty)$ and $\bar\epsilon \in (0, \doublebar\epsilon]$ to be determined during the course of the proof.
    We thus let $(\alpha^k, \gamma^k, \epsilon^k) \in \SPACEalphaCompactPos \times [\bar\gamma,\infty] \times [0,\bar\epsilon]$.
    Since this set is compact, we may assume that $\alpha^k \to \hat\alpha \in \SPACEalphaCompactPos$, and $\epsilon^k \to \hat\epsilon$, and $\gamma^k \to \hat\gamma$.
    %
    %
    Suppose $\hat \alpha \in \BD \SPACEalphaCompactPos$.
    By Lemma \ref{lemma:alpha-to-zero} then $Ku^k \to f$ strongly in $Y$ for small enough $\bar\epsilon$, with no conditions on $\bar\gamma$. Further by the $(K, q)$-minimality of $\bar f$ and Lemma \ref{lemma:k-convergence} then
    \begin{equation}
        \label{eq:contradiction-tildeu}
        \costf(\bar f)
        \le
        \liminf_{k \to \infty} \costf(Ku^k).
    \end{equation}

    If we fix $\gamma^k \defeq \gamma$ and $\epsilon^k$, and pick $\{\alpha^k\}_{k=1}^\infty$ is a minimising sequence for \eqref{eq:learn-numerical-single}, we find that \eqref{eq:contradiction-tildeu} is in contradiction to \eqref{eq:costf-tildeu}. Necessarily then $\hat\alpha \in \SPACEalphaCompactPosInt$. By the lower semicontinuity result of Proposition \ref{prop:existence-compact}, $\hat\alpha$ therefore has to solve \eqref{eq:learn-numerical-single}. We have proved \ref{item:a-exist-zero}, because, if $\mathcal{K}$ did not exist, we could choose $\alpha^k \to \hat\alpha \in \BD\SPACEalphaPos$.

    If $\gamma^k \to \hat\gamma$, $\epsilon^k \to \hat \epsilon$, and $\alpha^k$ solves \eqref{eq:learn-numerical-single} for $(\gamma, \epsilon)=(\gamma^k, \epsilon^k)$, then $\alpha^k \in \SPACEalphaCompactPosInt$ by \ref{item:a-exist-zero}. Now \eqref{eq:contradiction-tildeu} is in contradiction to \eqref{eq:costf-tildeu-epsilon}.
    Therefore \ref{item:a-exist-barepsilon} holds if \eqref{eq:costf-tildeu-epsilon} holds.
    
    It remains to verify \eqref{eq:costf-tildeu-epsilon} for $\doublebar\epsilon>0$ small enough and $\doublebar\gamma>0$ large enough. By Lemma \ref{lemma:strong-approximation}, we may find a sequence $\epsilon^k \downto 0$ and $\gamma^k \upto \infty$ such that $Ku_{\tilde \alpha,\gamma^k,\epsilon^k} \to K\tilde u_{\tilde \alpha}$ for some $\tilde u_{\tilde \alpha} \in \argmin \Jorig[\tilde \alpha]{\freevar}$.
    Since $\Phi$ is strictly convex, and both $u_{\tilde\alpha}, \tilde u_{\tilde \alpha} \in \argmin_{u \in X} \Jorig[\tilde \alpha]{u}$, we find that $K \tilde u_{\tilde\alpha}=K u_{\tilde\alpha}$. Recalling the $(K,p)$-minimality and -co-coercivity at $(u_{\tilde \alpha}, K^* \varphi_{u_{\tilde \alpha}})$, Lemma \ref{lemma:k-convergence} and \eqref{eq:costf-tildeu} now yield
    \[
        \limsup_{k \to \infty} \costf(u_{\tilde \alpha,\gamma^k,\epsilon^k}) = \costf(u_{\tilde \alpha}) < F(\bar f).
    \]
    Since we may repeat the above arguments on arbitrary sequences $(\gamma^k, \epsilon^k) \to (\infty, 0)$, we conclude that \eqref{eq:costf-tildeu-epsilon} holds for small enough $\bar \epsilon>0$ and large enough $\bar\gamma>0$.
\end{proof}

We now show the $\Gamma$-convergence of the cost map, and as a consequence the outer semicontinuity of the solution map. For an introduction to $\Gamma$-convergence, we refer to \cite{braides2002gamma,maso1993introduction}.

\begin{proposition}[$\Gamma$-convergence of the cost map and continuity of the solution map]
    \label{prop:gamma-convergence}
    Suppose Assumption \ref{ass:a-k}, Assumption \ref{ass:phi}, and Assumption \ref{ass:sol-smoothness} hold along with \eqref{eq:interior-solution-compact-strong}.
    Suppose, moreover, that $F$ is $(K, p)$-co-coercive, and every solution $u_{\alpha,\gamma,\epsilon}$ to \eqref{eq:denoise-numerical-single} is $(K, p)$-minimal with $\alpha \in \mathcal{K}$ and $(\gamma,\epsilon) \in [\bar \gamma, \infty] \times [0, \bar\epsilon]$.
    Then
    \begin{equation}
        \label{eq:gamma}
        \ValM{\gamma',\epsilon'}|\mathcal{K} \,\mathop{\to}^\Gamma\, \ValM{\gamma,\epsilon}|\mathcal{K}
    \end{equation}
    when $(\gamma',\epsilon'), (\gamma,\epsilon) \in [\bar \gamma, \infty] \times [0, \bar\epsilon]$ and $\mathcal{K}$ is as in \eqref{eq:interior-solution-compact-strong}.
    Moreover, the solution map
    \[
        \SolM(\gamma, \epsilon) = \argmin_{\alpha \in \SPACEalphaCompactPos}~\ValM{\gamma,\epsilon}(\alpha)
    \]
    is outer semicontinuous within $[\bar \gamma, \infty] \times [0, \bar\epsilon]$.
\end{proposition}

\begin{proof}
    Lemma \ref{lemma:gamma-lower} shows the $\Gamma$-lower limit \eqref{eq:gamma-lower}.
    We still have to show the $\Gamma$-upper limit. This means that given $\hat \alpha \in \mathcal{K}$ and $(\gamma^k,\epsilon^k) \to (\gamma,\epsilon)$
    within $[0, \bar \epsilon] \times [\bar\gamma, \infty]$,
    we have to show the existence of a sequence $\{\alpha^k\}_{k=1}^\infty \subset \mathcal{K}$ such that
    \[
        \ValM{\gamma,\epsilon}(\hat \alpha)
        \ge \limsup_{k \to \infty}
        \ValM{\gamma^k,\epsilon^k}(\alpha^k).
    \]

    We claim that we can take $\alpha^k=\hat \alpha$. With $u^k \defeq u_{\alpha^k,\gamma^k,\epsilon^k}$, Lemma \ref{lemma:strong-approximation} gives a subsequence satisfying $Ku^k \to K\hat u$ strongly with $\hat u$ a minimiser of $\JepsilonXX[\alpha]{\freevar}{\gamma,\epsilon}$. We just have to show that
    \begin{equation}
        \label{eq:costf-limit-cont}
        \costf(\hat u)=\lim_{k \to \infty} \costf(u^k).
    \end{equation}
    Since $F$ is $(K,p)$-co-coercive, and $\hat u$ by assumption $(K,p)$-minimal, this follows from Lemma \ref{lemma:k-convergence}.

    We have therefore established the $\Gamma$-convergence of $\ValM{\gamma',\epsilon'}|\mathcal{K}$ to $\ValM{\gamma,\epsilon}|\mathcal{K}$ as $(\gamma',\epsilon') \to (\gamma,\epsilon)$ within $[\bar\gamma, \infty] \times [0, \bar \epsilon]$.
    Our assumption \eqref{eq:interior-solution-compact-strong} says that the family $\{\ValM{\gamma',\epsilon'} \mid (\gamma',\epsilon') \in [\bar\gamma, \infty] \times [0, \bar \epsilon]\}$ is equi-mildly coercive in the sense of \cite{braides2002gamma}. Therefore, by the properties of $\Gamma$-convergence, see \cite[Theorem 1.12]{braides2002gamma}, the solution map is outer semicontinuous.
\end{proof}

\subsection{The $L^2$-squared fidelity with $(K,2)$-co-coercive cost}

In what follows, we seek to prove \eqref{eq:costf-tildeu} for the $L^2$-squared fidelity with $(K, 2)$-co-coercive cost functionals by imposing more natural conditions derived from \eqref{eq:regfix-interior-condition-barf} in the next lemma.

\begin{lemma}[Natural conditions for $L^2$-squared $2$-co-coercive case]
    \label{lemma:cocoercive}
    Suppose Assumption \ref{ass:a-k} and \ref{ass:aj-zeroterm-approx} hold.
    Let $Y$ be a Hilbert space, $f \in \range{K}$, and
    \[
        \Phi(v)=\frac{1}{2} \norm{f-v}_Y^2.
    \] 
    Then \eqref{eq:costf-tildeu} holds if $\bar f$ is $(K,2)$-minimal, $\costf$ is $(K, 2)$-co-coercive at $(\bar f, K^*\varphi_{\bar f})$ with $K^*\varphi_{\bar f} \in \subdiff\costf(\bar f)$, and
    \begin{equation}
        \label{eq:regfix-interior-condition-barf}
        \regfMARG{f} > \regfMARG{f-t \varphi_{\bar f}}
    \end{equation}
    for some $\bar\alpha \in \SPACEalphaPosInt$ and $t \in (0, 1/C]$, where $C$ is the co-coercivity constant.
\end{lemma}

Here we recall the definition of $\regfMARG{\freevar}$ from \eqref{eq:regfmarg}.

\begin{proof}
    Let $\alpha \in \SPACEalphaPosInt$. We have from the co-coercivity \eqref{eq:co-coercive} that
    \[
        \costf(\bar f) - \costf(u_\alpha) \ge -\dprod{K^* \varphi_{\bar f}}{u_\alpha-\bar f} - \frac{C}{2}\norm{K u_\alpha-f}_Y^2.
    \]
    Using the definition of the subdifferential, 
    \[
        \costf(u)-\costf(\bar f) \ge \dprod{K^* \varphi_{\bar f}}{u-\bar f}, \quad (u \in X).
    \]
    Summing, therefore
    \begin{equation}
        \label{eq:costf-l2l2-est0.5}
        \costf(u)-\costf(u_\alpha) \ge \dprod{K^*\varphi_{\bar f}}{u-u_\alpha} - \frac{C}{2}\norm{K u_\alpha-f}_Y^2, \quad (u \in X).
    \end{equation}
    Setting $u=\bar f$, we deduce
    \begin{equation}
        \label{eq:costf-l2l2-est1}
        \costf(\bar f)-\costf(u_\alpha) \ge \dprod{\varphi_{\bar f}}{f-K u_\alpha} - \frac{C}{2}\norm{K u_\alpha-f}_Y^2.
    \end{equation}

    Let $\alpha=t\bar\alpha$ for some $t>0$.
    Since $\Phi \circ K$ is continuous with $\Dom (\Phi \circ K)=X$, the optimality conditions for $u_\alpha$ solving \eqref{eq:denoise-numerical-single} state \cite[Proposition I.5.6]{ekeland1999convex}
    \begin{equation}
        \label{eq:subproblem-opt-l2}
        0 \in K^* (K u_{t\bar\alpha} - f) + t A^* [\subdiff \regforig[\bar\alpha]{\freevar}](A u_{t\bar\alpha}).
    \end{equation}
    Because $u_{t\bar\alpha}$ solves \eqref{eq:denoise-numerical-single}, we have $\regforig[\bar\alpha]{Au_{t\bar\alpha}}=\regfMARG{Ku_{t\bar\alpha}}$.
    By Lemma \ref{lemma:regfmarg-subdiff} below, therefore
    \begin{equation}
        \label{eq:subproblem-opt-l2-t}
        0 \in K^* (K u_{t\bar\alpha} - f) + t \psi_t,
        \quad
        \psi_t \in [\subdiff \regfMARG[\bar\alpha]{K\freevar}](u_{t\bar\alpha}).
    \end{equation}
    Multiplying by $\barK$ we deduce $f-Ku_{t\bar\alpha} = t \barK \psi_t$, so that referring back to \eqref{eq:costf-l2l2-est1}, and using the definition of the subdifferential, we get for any $t > 0$ the estimate
    \begin{equation}
        \notag
        \begin{split}
        \costf(\bar f)-\costf(u_{t\bar\alpha}) 
        &
        \ge
        \dprod{t\invK\varphi_{\bar f}}{\psi_t}
        - \frac{C}{2}\norm{Ku_{t\bar\alpha}-f}_Y^2
        \\
        &
        =
        \dprod{u_{t\bar\alpha}-(\bar f-t\invK\varphi_{\bar f})}{\psi_t}
        +
        \dprod{\bar f-u_{t\bar\alpha}}{\psi_t}
        - \frac{C}{2}\norm{Ku_{t\bar\alpha}-f}_Y^2
        \\
        &
        \ge \regfMARG[\bar\alpha]{K u_{t\bar\alpha}} 
        - \regfMARG[\bar\alpha]{K(\bar f - t \invK \varphi_{\bar f})}
        + \dprod{\bar f-u_{t\bar\alpha}}{\psi_t}
        - \frac{C}{2}\norm{Ku_{t\bar\alpha}-f}_Y^2.
        \end{split}
    \end{equation}
    Since $u_{t\bar\alpha}$ solves \eqref{eq:denoise-numerical-single} for $\alpha=t\bar\alpha$, using \eqref{eq:subproblem-opt-l2-t}, we have
    \[
        \dprod{\bar f-u_{t\bar\alpha}}{\psi_t}
        =\frac{1}{t}\norm{Ku_{t\bar\alpha}-f}_Y^2
        \ge 2\left(\regfMARG[\bar\alpha]{f} - \regfMARG[\bar\alpha]{K u_{t\bar\alpha}}\right).
    \]
    It follows
    \begin{equation}
        \label{eq:cocost-est1}
        \costf(\bar f)-\costf(u_{t\bar\alpha}) 
        \ge \regfMARG[\bar\alpha]{K \bar f} 
        - \regfMARG[\bar\alpha]{K(\bar f - t \invK \varphi_{\bar f})}
        + \frac{\inv t -C}{2}\norm{Ku_{t\bar\alpha}-f}_Y^2.
    \end{equation}
    %
    %
    We see that \eqref{eq:regfix-interior-condition-barf} implies \eqref{eq:costf-tildeu} if $0 < t \le \inv C$.
\end{proof}

\begin{lemma}   
    \label{lemma:regfmarg-subdiff}
    Suppose $\psi \in [\subdiff \regforig[\bar\alpha]{\freevar}](Au)$ with $A^*\psi \in \range{K^*}$, and that $\regforig[\bar\alpha]{Au}=\regfMARG{Ku}$. Then
    $A^*\psi \in [\subdiff \regfMARG[\bar\alpha]{K\freevar}](u)$
\end{lemma}

\begin{proof}
    Let $\lambda \in Y$ be such that $K^*\lambda=A^*\psi$.
    By the definition of the subdifferential, we have
    \[
        \regforig[\bar\alpha]{Au''}-\regforig[\bar\alpha]{Au} \ge \iprod{\lambda}{K(u'-u)}, \quad (u'' \in X).
    \]
    Minimising over $u'' \in X$ with $Ku''=Ku'$ for some $u' \in X$, and using $\regforig[\bar\alpha]{Au}=\regfMARG{Ku}$, we deduce
    \[
        \regfMARG{Ku'}
        -
        \regfMARG{Ku} \ge \iprod{\lambda}{K(u'-u)}, \quad (u' \in X).
    \]
    Thus
    \[
        \regfMARG{Ku'}
        -
        \regfMARG{Ku} \ge \iprod{A^*\psi}{u'-u}, \quad (u' \in X).
    \]
    This proves the claim.
\end{proof}

Summarising the developments so far, we may state:

\begin{proposition}
    \label{prop:conditions-squared-squared}
    Suppose Assumption \ref{ass:a-k} hold \ref{ass:aj-zeroterm-approx}. 
    Let $Y$ be a Hilbert space, $f \in Y \isect \range{K}$, and
    \[
        \Phi(v)=\frac{1}{2} \norm{f-v}_Y^2.
    \]
    If $\bar f$ is $(K, 2)$-minimal, $\costf$ is $(K, 2)$-co-coercive at $(\bar f, K^*\varphi_{\bar f})$, and
    \begin{equation}
        \label{eq:regfix-interior-condition-huber}
        \regfMARG{f} > \regfMARG{f-t \varphi_{\bar f}}
    \end{equation}
    for some $\bar\alpha \in \SPACEalphaPosInt$ and $t \in (0, 1/C]$, then the claims of Proposition \ref{prop:conditions-onealpha} hold.
\end{proposition}

\begin{proof}
    It is easily checked that Assumption \ref{ass:phi} holds.
    Lemma \ref{lemma:cocoercive} then verifies the remaining conditions of Proposition \ref{prop:conditions-onealpha}, which shows the existence of $\mathcal{K}$ in both cases. Finally, Proposition \ref{prop:existence-compact} shows the existence of $\hat \alpha \in \SPACEalphaCompactPosInt$ solving \eqref{eq:learn-numerical-single}.
    For the continuity of the solution map, we refer to Proposition \ref{prop:gamma-convergence}.
\end{proof}

\subsection{$L^2$-squared fidelity with $L^2$-squared cost}

We may finally finish the proof of our main result on the $L^2$ fidelity $\Phi(v) \defeq \frac{1}{2}\norm{f-v}_Y^2$, $Y=\SPACEkuLtwo$, with the $L^2$-squared cost functional $F(u)=\frac{1}{2}\norm{K_0u-f_0}_Z^2$.

\begin{proof}[Proof of Theorem \ref{thm:l2cost-main}]
    We have to verify the conditions of Proposition \ref{prop:conditions-squared-squared}, primarily the $(K, 2)$-minimality of $\bar f$, the $K$-cocoercivity of $F$, and \eqref{eq:regfix-interior-condition-huber}.
    Regarding minimality and co-coercivity, we write $F=F_0 \circ K_0$, where $F_0(v)=\frac{1}{2}\norm{v-f_0}_Z^2$. Then for any $v, v' \in Z$, we have
    \[
        F_0(v')-F_0(v)
        =
        \iprod{v'-v}{v-f_0}
        +
        \frac{1}{2}\norm{v'-v}_{L^2(\Omega)}^2.
    \]    
    From this $(I, 2)$-co-coercivity of $F_0$ with $C=1$ is clear, as is the $(I, 2)$-minimality with regard to $F_0$ of every $v \in Y$.
    By extension, $F$ is easily seen to be $(K_0, 2)$-co-coercive, and every $u \in X$ $(K_0, 2)$-minimal. Using \eqref{eq:k0-condition}, $(K, 2)$-co-coercivity of $F$ with $C=C_0$ is immediate, as is the $(K, 2)$-minimality of every $u \in X$.
    
    Regarding \eqref{eq:regfix-interior-condition-huber}, we need to find $\varphi_{\bar f}$ such that $K^* \varphi_{\bar f}=\grad F(\bar f)$.
    We have
    \[
        K_0^*(K_0\bar f-f_0)=\grad F(\bar f).
    \]
    From this we observe that $\varphi_{\bar f}$ exists, because \eqref{eq:k0-condition}
    implies $\nullspace{K} \subset \nullspace{K_0}$, and hence $\range{K^*} \subset \range{K_0^*}$.
    Here $\mathcal{N}$ and $\mathcal{R}$ stand for the nullspace and range, respectively.
    Setting $K^* \varphi_{\bar f}=K_0^*(K_0\bar f-f_0)$ and using $K \invK=I$ on $\range{K}$, we thus find that
    \[
        \varphi_{\bar f}=(K_0 \invK)^*(K_0\bar f-f_0).
    \]
    Observe that since $\nullspace{K} \subset \nullspace{K_0}$, this expression does not depend on the choice of $\bar f \in \inv K f$. Following Remark \ref{remark:barf}, we can replace $\bar f=\invK f$. It follows that \eqref{eq:l2cost-interior-condition} implies \eqref{eq:regfix-interior-condition-huber}.
\end{proof}

\begin{remark}
    \label{remark:phi-pseudol2}
    Provided that $\Phi$ satisfies Assumption \ref{ass:phi}, \ref{ass:aj-zeroterm-approx}, and \ref{ass:sol-smoothness}, it is easy to extend Lemma \ref{lemma:cocoercive} and consequently Theorem \ref{thm:l2cost-main} to the case
    \[
        \Phi(v)=\frac{1}{2}\norm{v}_{\hat Y}^2 - \dprod{f}{v}_{Y^*,Y},
        \quad (v \in Y),
    \]
    where $\hat Y \supset Y$ is a Hilbert space, $f \in Y^*$, and $Y$ still a reflexive Banach space. As $Y \subset \hat Y=\hat Y^* \subset Y^*$, in this case, we still have
    \[
        \grad\Phi(v)=v-f \in Y^*.
    \]
    In particular
    \[
        \grad\Phi(Ku)=K^*(Ku-f) \in X^*.
    \]
    Therefore the expression \eqref{eq:subproblem-opt-l2} still holds, which is the only place where we needed the specific form of $\Phi$.
\end{remark}

\begin{example}[Bingham flow]
    \label{ex:bingham-condition}
    As a particular case of this remark, we take $\hat Y=Y=H_0^1(\Omega)$. Then $Y^*=H^{-1}(\Omega)$. With $f \in L^2(\Omega)$, the Riesz representation theorem allows us to write
    \[
        \int_\Omega f v \d x=\dprod{\tilde f}{v}_{H^{-1}(\Omega),H_0^1(\Omega)},
    \]
    for some $\tilde f \in H^{-1}(\Omega)$, which we may identify with $f$.
    Therefore, Theorem \ref{thm:l2cost-main} can be extended to cover the Bingham flow
    of Example \ref{ex:bingham}. In particular, we get the same condition for interior
    solutions as in Corollary \ref{corollary:tv-l2cost-interior}, namely
    \[
        \TV(f)
        >
        \TV(f_0).
    \]
\end{example}

\subsection{A more general technique for the $L^2$-squared fidelity}

We now study another technique that does not require $(K, 2)$-minimality and $(K, 2)$-co-coercivity at $\bar f$. We still however require $\Phi$ to be the $L^2$-squared fidelity, and $u_{\doublebar\alpha}$ to be $(K,p)$-minimal.

\begin{lemma}[Natural conditions for the general $L^2$-squared case]
    \label{lemma:conditions-general-squared}
    Suppose Assumption \ref{ass:a-k} and \ref{ass:aj-zeroterm-approx} hold.
    Let $Y$ a Hilbert space, $f \in Y \isect \range{K}$, and
    \[
        \Phi(v)=\frac{1}{2} \norm{f-v}_Y^2.
    \]
    The claims of Proposition \ref{prop:conditions-onealpha} hold if for some $\bar\alpha \in \SPACEalphaPosInt$ and $t > 0$ both $\bar f$ and the solution $u_{\bar \alpha}$ are $(K, p)$-minimal and
    \begin{equation}
        \label{eq:general-interior-ass2t}
        \regfMARG[\bar\alpha]{Ku_{\bar \alpha}} > \regfMARG[\bar\alpha]{f-t \varphi_{u_{\bar\alpha}}}.
    \end{equation}
\end{lemma}

\begin{remark}
    Setting $\bar \alpha=s\doublebar\alpha$ in \eqref{eq:general-interior-ass2t}, we see employing the lower semicontinuity Lemma \ref{lemma:t-lsc} that the former is implied by
    \begin{equation}
        \label{eq:regfix-interior-condition-limiting}
        \regfMARG{f} > \limsup_{s \downto 0} \regfMARG{f-t \varphi_{s \doublebar\alpha}}.
    \end{equation}
    Here we use the shorthand $\varphi_{s \doublebar\alpha} \defeq \varphi_{u_{s \doublebar\alpha}}$.
    The difficulty is going to the limit, because we do not generally have any reasonable form of convergence of $\{\varphi_{s\doublebar\alpha}\}_{s > 0}$. If we did indeed have $\varphi_{s\doublebar\alpha} \to \varphi_{\doublebar f}$, then \eqref{eq:regfix-interior-condition-limiting} and consequently \eqref{eq:general-interior-ass2t-summary} would be implied by the condition \eqref{eq:regfix-interior-condition-huber} we derived using $(K, 2)$-co-coercivity.
    We will in the next subsection go to the limit with finite-dimensional functionals that are not $(K, 2)$-co-coercive and hence the earlier theory does not apply.
\end{remark}

\begin{proof}
    Let us observe that \eqref{eq:costf-tildeu} holds if for some $\bar \alpha > 0$ and $C>0$, we can find a $(K, p)$-minimal
    \[
        u_{\bar \alpha} \in \argmin_{u \in X} \Jorig[\bar \alpha]{u},
    \]
    satisfying
    \begin{equation}
        \label{eq:general-interior-ass0}
        \dprod{\varphi_{{\bar \alpha}}}{f-Ku_{\bar \alpha}} > 0.
    \end{equation}
    Here we denote for short $\varphi_{\bar \alpha} \defeq \varphi_{u_{\bar \alpha}}$, recalling that $K^* \varphi_{\bar \alpha} \in \subdiff F(u_{\bar \alpha})$.
    Indeed, by the definition of the subdifferential, the minimality of $u_{\bar \alpha}$, and \eqref{eq:general-interior-ass0}, we deduce
    \begin{equation}
        \label{eq:cost-ualpha-est}
        \costf(\bar f) \ge \costf(u_{\bar \alpha}) + \dprod{\varphi_{{\bar \alpha}}}{f-K u_{\bar \alpha}}
        >
        \costf(u_{\bar \alpha}).
    \end{equation}
    This shows \eqref{eq:costf-tildeu}.
    
    We need to show that that \eqref{eq:general-interior-ass2t} implies \eqref{eq:general-interior-ass0}.
    As in the proof of Lemma \ref{lemma:cocoercive}, we deduce by application of Lemma \ref{lemma:regfmarg-subdiff} that
    \begin{equation}
        \label{eq:psi-alpha-eq}
        0 \in K^* (K u_{\bar\alpha} - f) + \psi_{\bar\alpha},
        \quad\text{for some}\quad
        \psi_{\bar\alpha} \in [\subdiff \regfMARG[\bar\alpha]{K\freevar}](u_{\bar\alpha}).
    \end{equation}
    Then 
    \[
        \dprod{\bar f-u_{\bar \alpha}}{\psi_{\bar \alpha}}
        =
        \dprod{\bar f- u_{\bar \alpha}}{K^*(f-K u_{\bar \alpha})}
        =
        \norm{f- K u_{\bar \alpha}}_Y^2 \ge 0.
    \]
    Multiplying \eqref{eq:psi-alpha-eq} by $\barK$ and using this estimate, we deduce for any $t>0$ that
    \[
        \begin{split}
        \dprod{\varphi_{\bar \alpha}}{f-Ku_{\bar \alpha}}
        &
        \ge
        \inv t \dprod{u_{\bar \alpha}-(u_{\bar \alpha}-t \invK \varphi_{\bar \alpha})}{\psi_{\bar \alpha}}
        \\
        & =
        \inv t
        \dprod{u_{\bar \alpha}-(\bar f- t \invK \varphi_{\bar \alpha})}{\psi_{\bar \alpha}}
        +\inv t \dprod{\bar f-u_{\bar \alpha}}{\psi_{\bar \alpha}}
        \\
        &
        \ge
        \inv t 
        \dprod{u_{\bar \alpha}-(\bar f- t \invK \varphi_{\bar \alpha})}{\psi_{\bar \alpha}}
        \\
        &
        \ge \inv t {\bar \alpha}
        \left(
            \regfMARG[\bar\alpha]{Ku_{\bar \alpha}}-\regfMARG[\bar\alpha]{K(\bar f-t\invK \varphi_{\bar \alpha})}
        \right).
        \end{split}
    \]
    The last step follows from the definition of $\subdiff \regfMARG[\bar\alpha]{K\freevar}$. This proves that \eqref{eq:general-interior-ass2t} implies \eqref{eq:general-interior-ass0}. 
\end{proof}

Summing up the developments so far, we may in contrast to Proposition \ref{prop:conditions-squared-squared} that depended on $\bar f$ and co-coercivity, state:

\begin{proposition}
    \label{prop:conditions-general-squared}
    Suppose Assumption \ref{ass:a-k} and \ref{ass:aj-zeroterm-approx} hold. 
    Let $Y$ be a Hilbert space, $f \in Y \isect \range{K}$, and
    \[
        \Phi(v)=\frac{1}{2} \norm{f-v}_Y^2.
    \]
    If for some $\doublebar\alpha \in \SPACEalphaPosInt$, $t>0$, the solution $u_{\doublebar\alpha}$ is $(K,p)$-minimal with
    \begin{equation}
        \label{eq:general-interior-ass2t-summary}
        \regfMARG[\doublebar\alpha]{Ku_{\doublebar \alpha}} > \regfMARG[\doublebar\alpha]{f-t \varphi_{u_{\doublebar \alpha}}},
    \end{equation}
    then there exist $\bar\gamma >0$ and $\bar\epsilon>0$ such that the following hold.
    \begin{enumroman}
        \item
        For each $\epsilon \in [0, \bar \epsilon]$ and $\gamma \in [\bar\gamma,\infty]$ there exists a compact set $\mathcal{K} \subset \SPACEalphaCompactPosInt$ such that \eqref{eq:interior-solution-compact} holds. In particular there exists a solution $\hat \alpha \in \SPACEalphaCompactPosInt$ to \eqref{eq:learn-numerical-single}.
        \item
        If, moreover,  Assumption \ref{ass:sol-smoothness} holds, there exists a compact set $\mathcal{K} \subset \SPACEalphaCompactPosInt$ such that \eqref{eq:interior-solution-compact-strong} holds and the solution map $\SolM$ is outer semicontinuous within $[\bar\gamma,\infty] \times [0, \bar \epsilon]$.
    \end{enumroman}
\end{proposition}

\begin{proof}
    It is easily checked that Assumption \ref{ass:phi} holds.
    Lemma \ref{lemma:conditions-general-squared} then verifies the remaining conditions of Proposition \ref{prop:conditions-onealpha}, which shows the existence of $\mathcal{K}$ in both cases. Finally, Proposition \ref{prop:existence-compact} shows the existence of $\hat \alpha \in \SPACEalphaCompactPosInt$ solving \eqref{eq:learn-numerical-single}.
    For the continuity of the solution map, we refer to Proposition \ref{prop:gamma-convergence}.
\end{proof}

\subsection{$L^2$-squared fidelity with Huberised $L^1$-type cost} 

We now study the Huberised total variation cost functional. We cannot in general prove that solutions $u_\alpha$ for small $\alpha$ are better than $f$. Consider, for example $f_0$ a step function, and $f$ a noisy version without the edge destroyed. The solution
$u_\alpha$ \emph{might} smooth out the edge, and then we might have $\radon{Du_\alpha-Df} \approx \radon{Du_\alpha}+\radon{Df} > \radon{Df_0-Df} \approx 0$. This destroys all hope of verifying the conditions of Lemma \ref{lemma:conditions-general-squared} in the general case. 
If we however modify the set of test functions in the definition of $\costhubertv$ to be discrete we can prove this bound. Alternatively, we could assume uniformly bounded divergence from the family of test functions. We have left this case out for simplicity, and prove our results for general $L^1$ costs with finite-dimensional $Z$.

\begin{lemma}[Conditions for $\costlone$ cost]
    \label{lemma:modcost}
    Suppose $Z$ is a reflexive Banach space, and $K_0: X \to Z$ is linear and bounded, and satisfies \eqref{eq:k0-condition}.
    Then, whenever \eqref{eq:regfix-interior-condition-huber} holds, $\costf(u) \defeq \costf_{\costlone}(K_0 u)$ is $(K, 1)$-co-coercive and \eqref{eq:general-interior-ass2t-summary} holds for some $\doublebar\alpha \in \SPACEalphaPosInt$ with both $u_{\doublebar\alpha}$ and $\bar f$ being $(K, 1)$-minimal.
\end{lemma}

\begin{proof}
    Denote
    \[
        B \defeq \{ \lambda \in Z^* \mid \norm{\lambda}_{Z^*} \le 1 \}.
    \]
    We first verify $(I, 1)$-co-coercivity of $\costf_{\costlone}$.
    Let $v, v^* \in Z$, and $\lambda^* \in B$ be such that $\lambda^* \in \subdiff \costf_{\costlone}(v^*)$.
    Clearly $\lambda^*$ achieves the maximum for $\costf_{\costlone}(v^*)$.
    Let $\lambda \in B$ achieve the maximum for $\costf_{\costlone}(v)$.
    Then
    \begin{equation}
        \label{eq:l1-1-coco}
        \begin{split}
        \costf_{\costlone}(v) - \costf_{\costlone}(v^*) 
        &
        \le
        \dprod{\lambda}{v-f_0} - \dprod{\lambda}{v^*-f_0}
        \\
        &
        =
        \dprod{\lambda^*}{v-v^*}
        +
        \dprod{\lambda-\lambda^*}{v-v^*}
        \\
        &
        \le
        \dprod{\lambda^*}{v-v^*}
        +
        2 \sup_{\lambda' \in B} \norm{\lambda'}\norm{v-v^*}
        \\
        &
        \le
        \dprod{\lambda^*}{v-v^*}
        +
        2 \norm{v-v^*}.
        \end{split}
    \end{equation}
    This proves $(I, 1)$-co-coercivity of $\costf_{\costlone}$. $(K, 1)$-co-coercivity of $F$ now follows similarly to the argument in the proof of Theorem \ref{thm:l2cost-main}, using \eqref{eq:k0-condition}.

    Analogously, taking the traingle inequality in \eqref{eq:l1-1-coco} in the opposite direction, we show that every $u \in X$ is $(K, 1)$-minimal.
    Therefore, in particular both $\bar f$ and $u_{\doublebar\alpha}$ are $(K, 1)$-minimal

    To verify \eqref{eq:general-interior-ass2t-summary}, it is enough to verify \eqref{eq:regfix-interior-condition-limiting}.
    Similarly to the proof of Theorem \ref{thm:l2cost-main}, using \eqref{eq:k0-condition}, we
    verify that $K^* \varphi_{s\bar\alpha} \in \subdiff \costf(u_{s\bar\alpha})$ exists, and
    \[
        \varphi_{s\bar\alpha} \in \barK \subdiff \costf(u_{s\bar\alpha}) = (K_0\invK)^* \lambda_{s\bar\alpha},
    \]
    where $\lambda_{s\bar\alpha} \in B$ achieves the maximum for $\costf_0(K_0 u_{s\bar\alpha}-f_0)$. In fact, $\varphi_{s\bar\alpha} \in \range{K}$. If this would not hold, we could find $v \perp \range{K}$ such that
    \[
            0 
            < \dprod{v}{\varphi_{s\bar\alpha}}
            = \dprod{K_0 \invK v}{\lambda_{s\bar\alpha}}
            \le \norm{K_0 \invK v}\norm{\lambda_{s\bar\alpha}}
            \le C \norm{K \invK v}\norm{\lambda_{s\bar\alpha}}.
    \]
    But, for any $v' \in \range{K}$,
    \[
        \dprod{v'}{K\invK v}=\dprod{(K\invK)^* v'}{v}=\dprod{v'}{v}=0.
    \]
    Therefore $\norm{K \invK v}=0$, and we reach a contradiction unless $\lambda_{s\bar\alpha}=0$, that is $\varphi_{s\bar\alpha}=0 \in \range{K}$.

    As is easily verified, $\subdiff F_{\costlone}$ is outer semicontinuous with respect to strong convergence in the domain $Z$ and weak* convergence in the codomain $Z^*$. 
    That is, given $v^k \to v$ and $z^k \weaktostar z$ with $z^k \in \subdiff F_{\costlone}(u^k)$, we have $z \in \subdiff F_{\costlone}(v)$.
    By Lemma \ref{lemma:alpha-to-zero} and \eqref{eq:k0-condition}, we have $K_0 u^k \to K_0 \bar f$ strongly in $Z$.
    Since $B$ is bounded, we may therefore find a sequence $s^k \downto 0$ with $\lambda_{s^k\bar\alpha} \weaktostar \lambda_0 \in \subdiff \costf(\bar f)$ weakly* in $Z^*$.
    Since by assumption $K_0$ is compact, then also $K_0^*$ is compact \cite[Theorem 4.19]{rudin2006functional}.
    Consequently $\varphi_{s^k\bar\alpha} \to \varphi_0 \defeq (K_0\invK)^* \lambda_0$
    strongly in $Y$ after possibly moving to an unrelabelled subsequence.
    Let us now consider the right hand side of \eqref{eq:general-interior-ass2t-summary} for $\doublebar\alpha=s^k\bar\alpha$. Since $f=K \bar f$, and we have proved that $\varphi_{u_{s^k\bar \alpha}} \in \range{K}$, Lemma \ref{lemma:t-lsc} shows that
    \begin{equation}
        \notag
        \lim_{k \to 0} \regfMARG[\bar\alpha]{f-t \varphi_{u_{s^k\bar \alpha}}}
        = \regfMARG[\bar\alpha]{f-t \varphi_0}.
    \end{equation}
    Minding the discussion surrounding \eqref{eq:regfix-interior-condition-limiting}, we observe that choosing $\doublebar\alpha=s^k\bar\alpha$ for large enough $k>0$, \eqref{eq:regfix-interior-condition-limiting} is implied by \eqref{eq:regfix-interior-condition-huber}. 
\end{proof}

\begin{proof}[Proof of Theorem \ref{thm:hubercost-main}]
    From the proof of Lemma \ref{lemma:modcost}, we observe that
    \eqref{eq:regfix-interior-condition-huber}, can be expanded as
    \begin{equation}
        \notag
        \regfMARG{f} > \regfMARG{f - t (K_0\invK)^* \lambda_0},
    \end{equation}
    where $\lambda_0 \in V$ with $\lambda_0 \in \subdiff \costf_{\costlone}(K_0 \bar f)$. 
    As in the proof of Theorem \ref{thm:l2cost-main}, this is in fact independent of the choice of $\bar f$, so may replace $\bar f = \invK f$.
    Thus $\lambda_0 \in \subdiff \costf_{\costlone}(K_0 \invK f)$.
    By Lemma \ref{lemma:modcost}, the conditions of Proposition \ref{prop:conditions-general-squared} are satisfied, so we may apply it together with Proposition \ref{prop:existence-compact} to conclude the proof.
\end{proof}

\begin{remark}
    The considerations of Remark \ref{remark:phi-pseudol2} also apply to Lemma \ref{lemma:conditions-general-squared} and consequently Theorem \ref{thm:hubercost-main}.
    That is, the results hold for the cost
    \begin{equation}
        \label{eq:phi-pseudo-l2}
        \Phi(v)=\frac{1}{2}\norm{v}_{\hat Y}^2 - \dprod{f}{v}_{Y^*,Y},
        \quad (v \in Y),
    \end{equation}
    where $\hat Y \supset Y$ is a Hilbert space, $f \in Y^*$, and $Y$ a reflexive Banach space.
    Indeed, again the specific form of $\Phi$ was only used for the optimality condition \eqref{eq:psi-alpha-eq}, which is also satisfied by the form \eqref{eq:phi-pseudo-l2}.
\end{remark}

\section{Numerical verification and insight}\label{sec:numerics}

In order to verify the above theoretical results, and to gain
further insight into the cost map $\ValM{\gamma,\epsilon}$, we computed the values for a grid of values of $\alphavec$, for both $\TV$ and $\TGV^2$ denoising, and $\costltwo$ and $\costhubertv$ cost functionals.
This we did for two different images, the parrot image depicted in Figure \ref{fig:dataset2} and the Scottish southern uplands image depicted in Figure \ref{fig:dataset4}.
The results are visualised in Figure \ref{fig:landscape-tv} and Figure \ref{fig:landscape-tgv2}, respectively.
For $\TV$, the parameter range was 
\[
    \alpha \in U \defeq \{0.001,0.01, 0.02, \ldots 0.5\}/\DIMdomain
\]
(altogether 51 values), where $\DIMdomain=256$ is the edge length of the rectangular test image. For $\TGV^2$ the parameter range was $\alphavec \in U \times (U/\DIMdomain)$. We set $\gamma=100$, $\epsilon=1\ee^{-10}$, and computed the denoised image $u_{\alpha,\gamma,\epsilon}$ by the SSN denoising algorithm that we report separately in \cite{tuomov-tgvlearn} with more extensive numerical comparisons and further applications.

As we can see, the optimal $\alphavec$ clearly seems to
lie away from the boundary of the parameter domain
$\SPACEalphaPos$, confirming the theoretical studies
for the squared $L^2$ cost $\costltwo$, and the
discrete version of the Huberised $\TV$ cost $\costhubertv$.
The question remains: do these results hold for the full
Huberised $\TV$?

We further observe from the numerical landscapes that the
cost map $\ValM{\gamma,\epsilon}$ is roughtly quasiconvex in the
variable $\alpha$ for both $\TV$ and $\TGV^2$. In the
$\beta$ variable of $\TGV^2$ the same does not seem to hold,
as around the optimal solutoin the level sets tend to expand 
along $\alpha$ as $\beta$ increases, until starting to reach 
their limit along $\beta$. However, the level sets around the
optimal solution also tend to be very elongated on the
$\beta$ axes. This suggests that $\TGV^2$ is reasonably robust
with respect to choice of $\beta$, as long as it is in the
right range.

\begin{figure}[t]
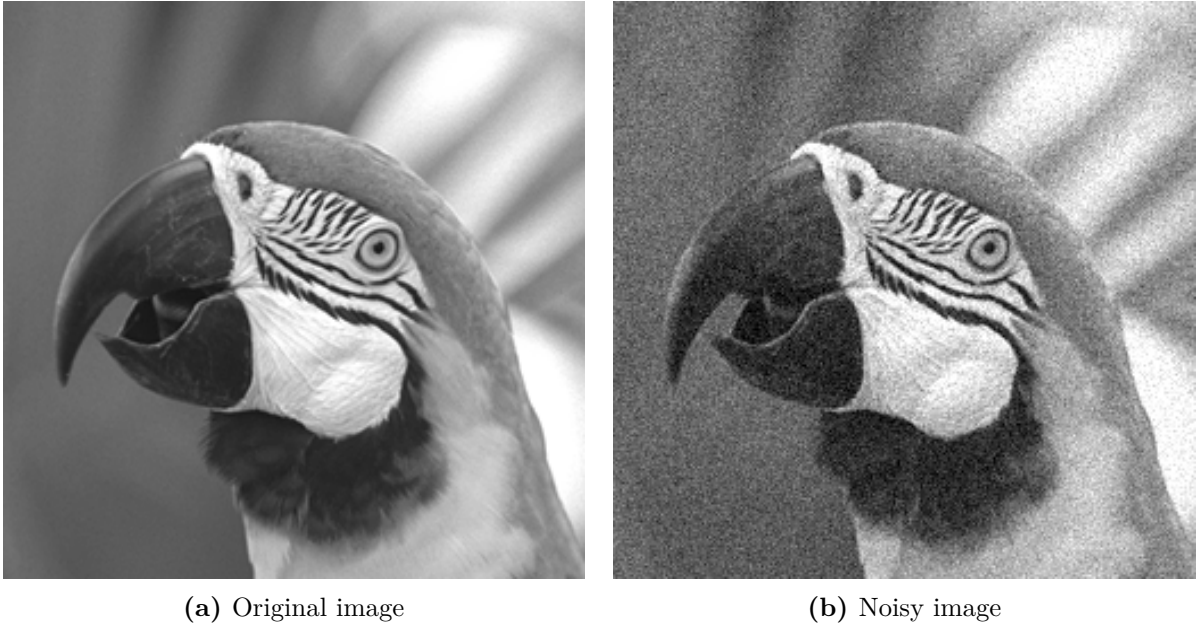

    \centering
    \def\subfigprefix{fig:dataset2}
    \inplot{kodim23gray-crop}
            {Original image}
    \inplot{kodim23gray-crop-noisy}
            {Noisy image}
    \caption{Parrot test image}
    \label{\subfigprefix}
\end{figure}

\begin{figure}[t]
    \centering
    \def\subfigprefix{fig:dataset4}
    \inplot{uplands}
            {Original image}
    \inplot{uplands-noisy}
            {Noisy image}
    \caption{Uplands test image}
    \label{\subfigprefix}
\end{figure}

\begin{figure}[t]
    \centering
    \begin{subfigure}[t]{0.47\textwidth}
        \includegraphics[width=\textwidth]{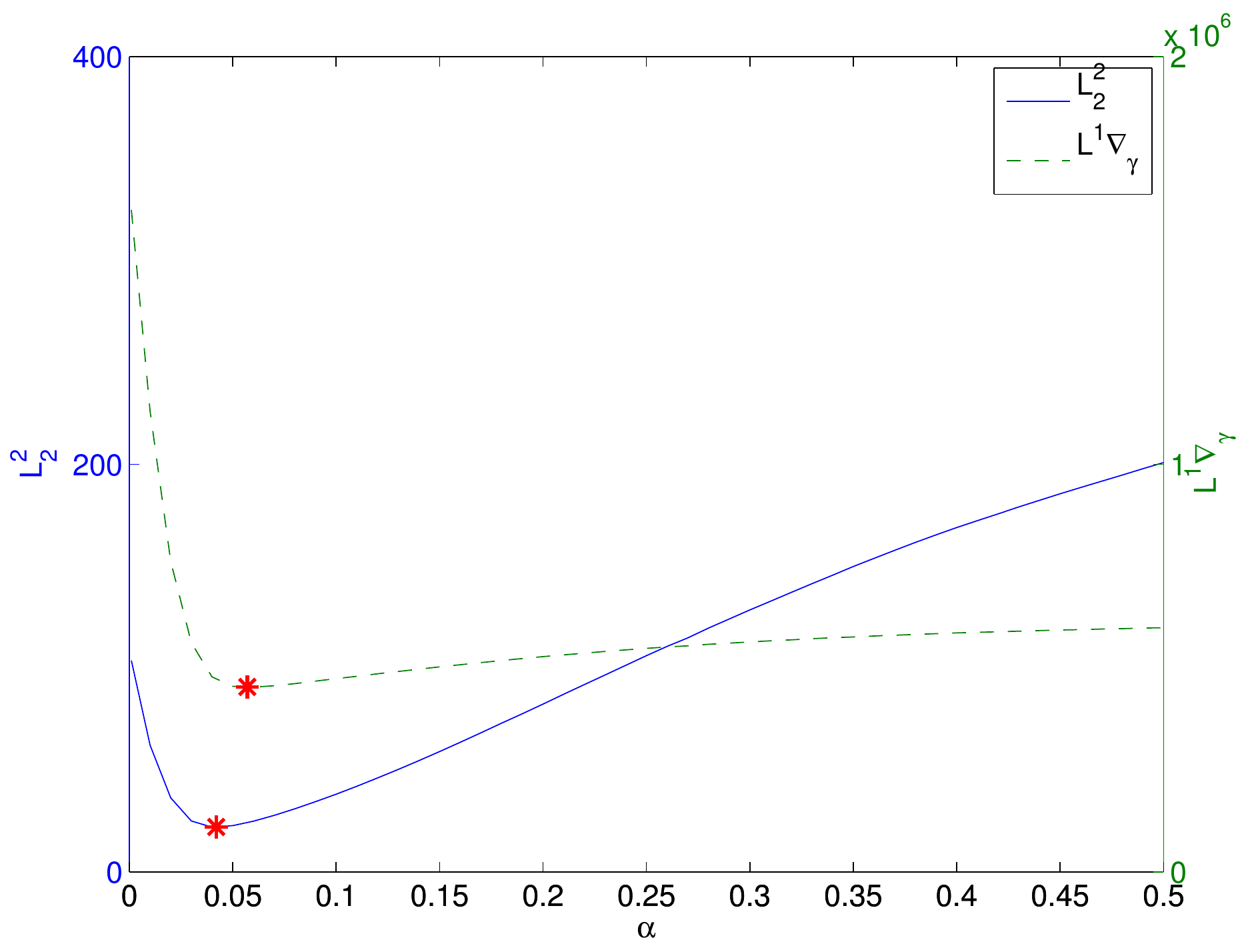}
        \caption{Parrot, $\TV$, all cost functionals}
    \end{subfigure}
    \begin{subfigure}[t]{0.47\textwidth}
        \includegraphics[width=\textwidth]{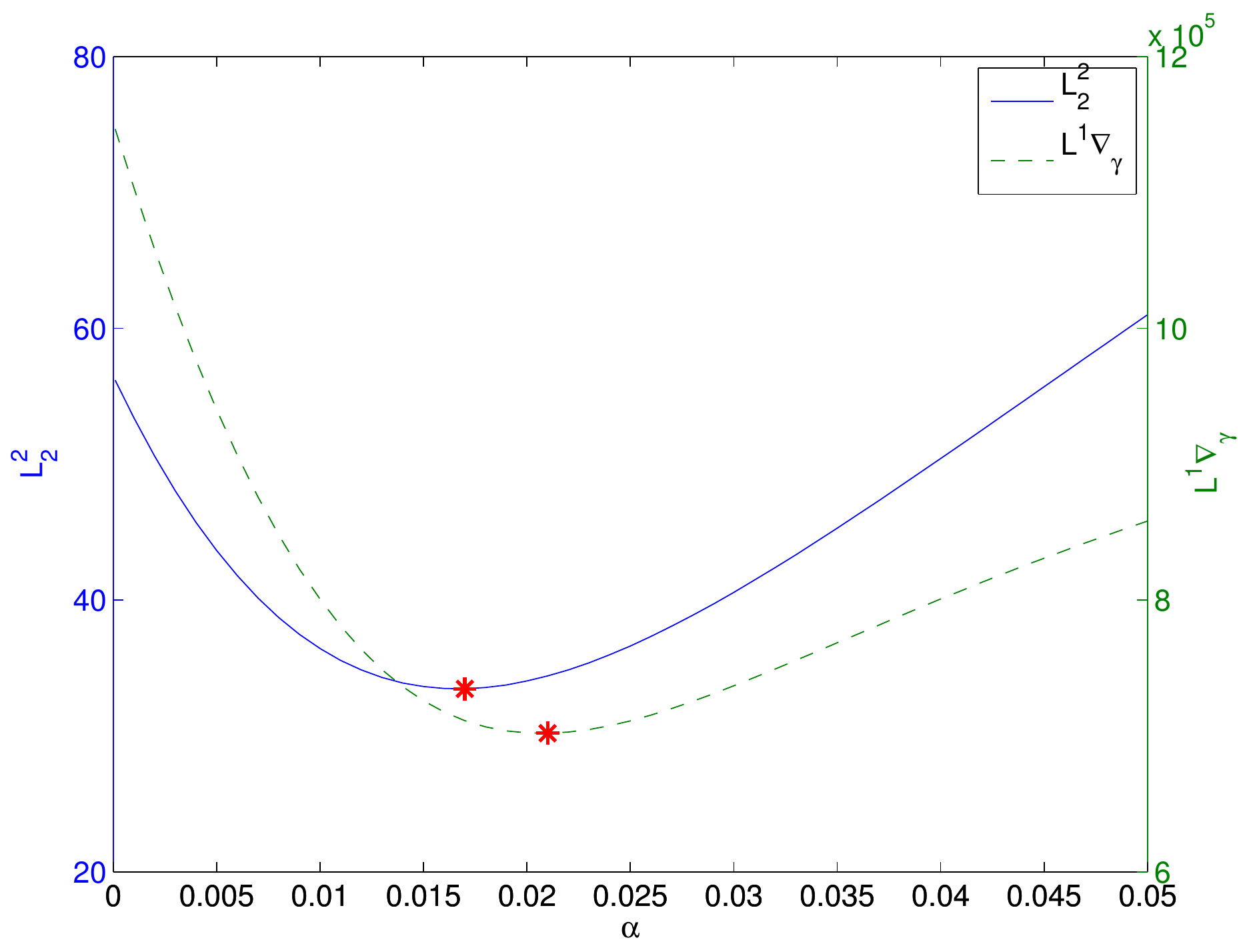}
        \caption{Uplands, $\TV$, all cost functionals}
    \end{subfigure}
    \caption{Cost functional value versus $\alpha$ for $\TV$ denoising, 
        for both the parrot and uplands test images, 
        for both {\costltwo} and {\costhubertv}
        cost functionals.
        For the parrot image, optimal $\alpha$ from 
        \cite{tuomov-tgvlearn}
        for the initialisation $0.1/\DIMdomain$, resp.~$(1/\DIMdomain^2, 0.1/\DIMdomain)$, is indicated by an asterisk.
        For the landscape image, optimal $\alpha$ from 
        \cite{tuomov-tgvlearn}
        for the initialisation $0.01/\DIMdomain$, resp.~$(0.1/\DIMdomain^2, 0.01/\DIMdomain)$, is indicated by an asterisk.
        }
    \label{fig:landscape-tv}
\end{figure}

\begin{figure}[t]
    \begin{subfigure}[t]{0.47\textwidth}
        \includegraphics[width=\textwidth]{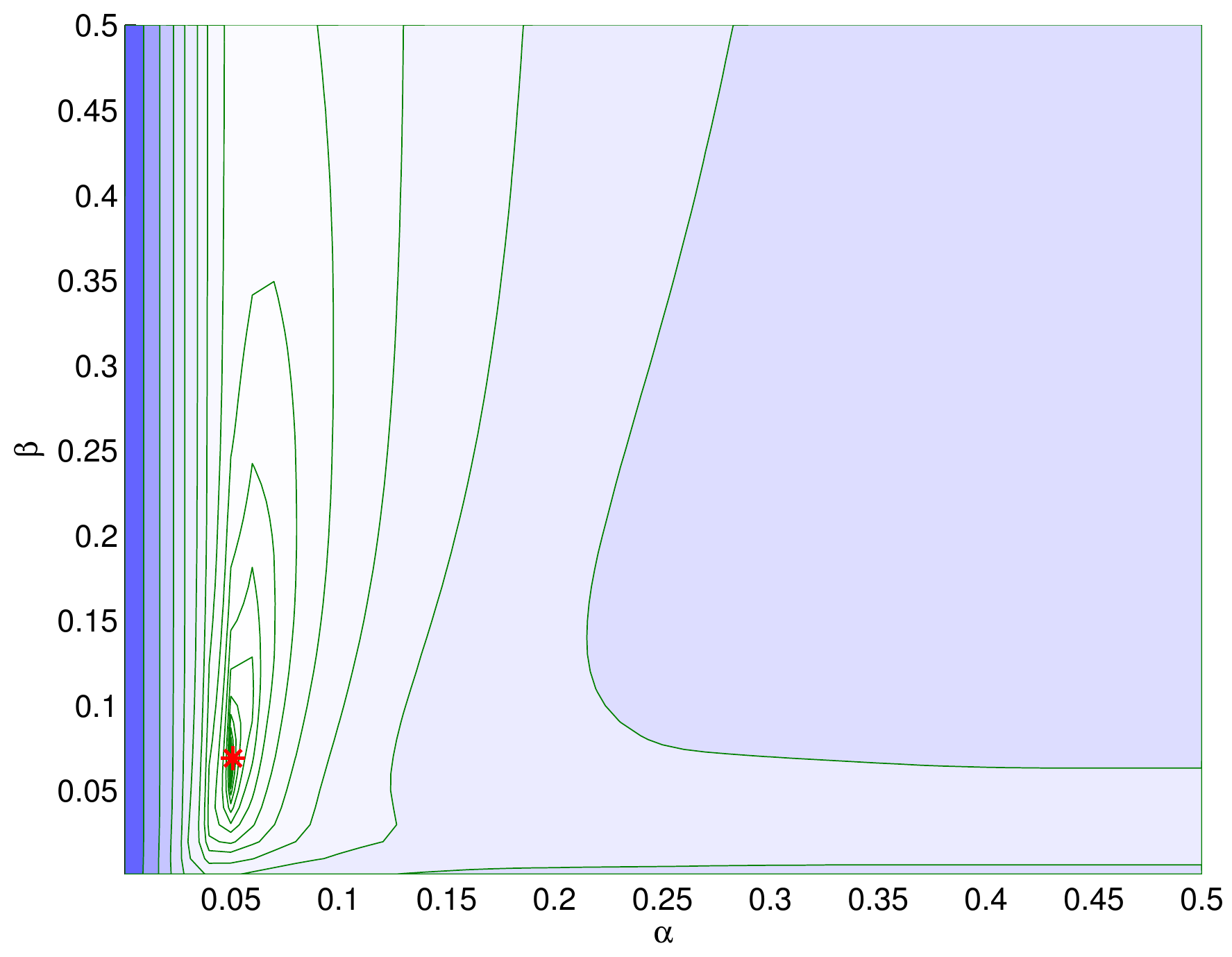}
        \caption{Parrot, $\TGV^2$, {\costhubertv} cost functional}
    \end{subfigure}
    \hfill
    \begin{subfigure}[t]{0.47\textwidth}
        \includegraphics[width=\textwidth]{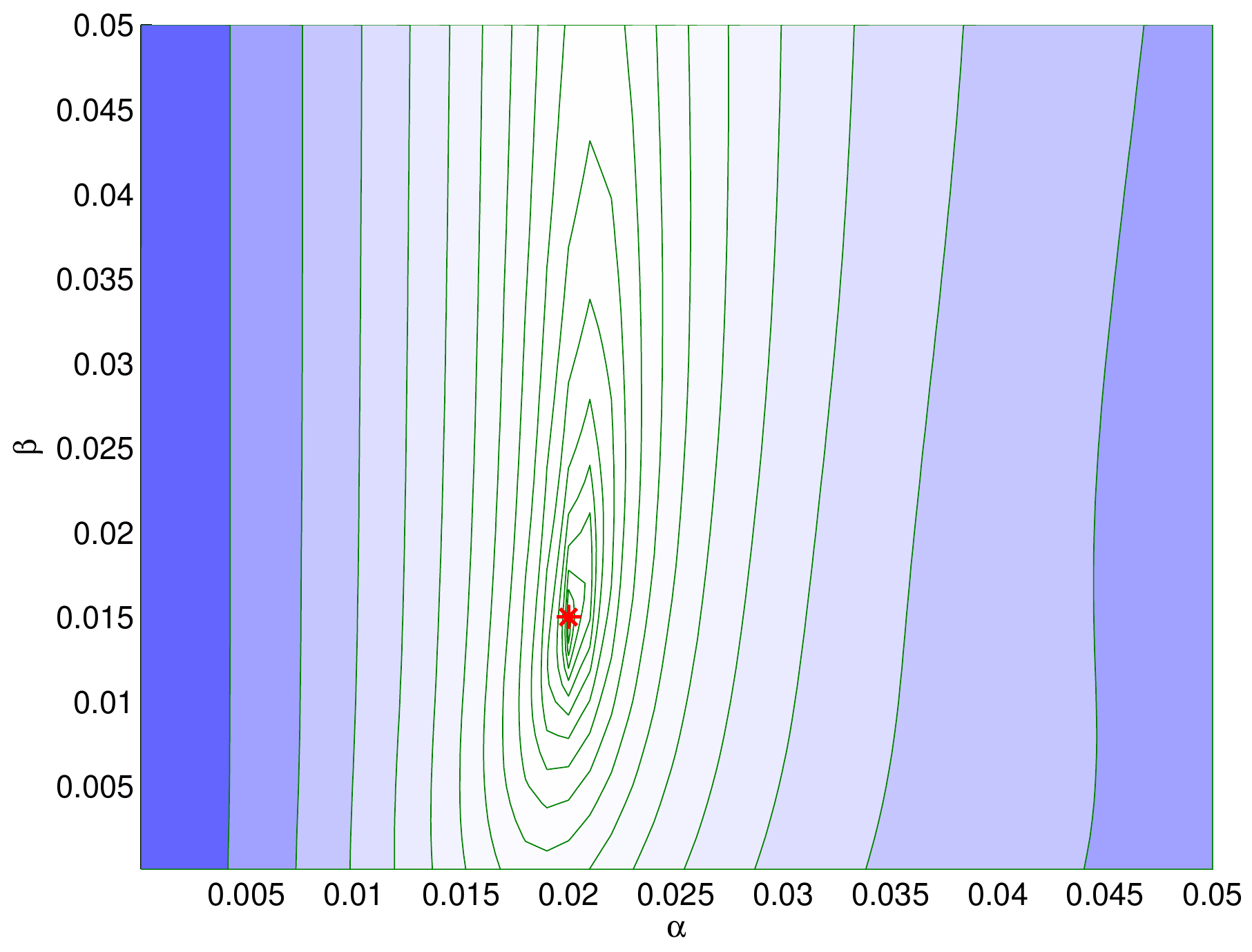}
        \caption{Uplands, $\TGV^2$, {\costhubertv} cost functional}
    \end{subfigure}
    \begin{subfigure}[t]{0.47\textwidth}
        \includegraphics[width=\textwidth]{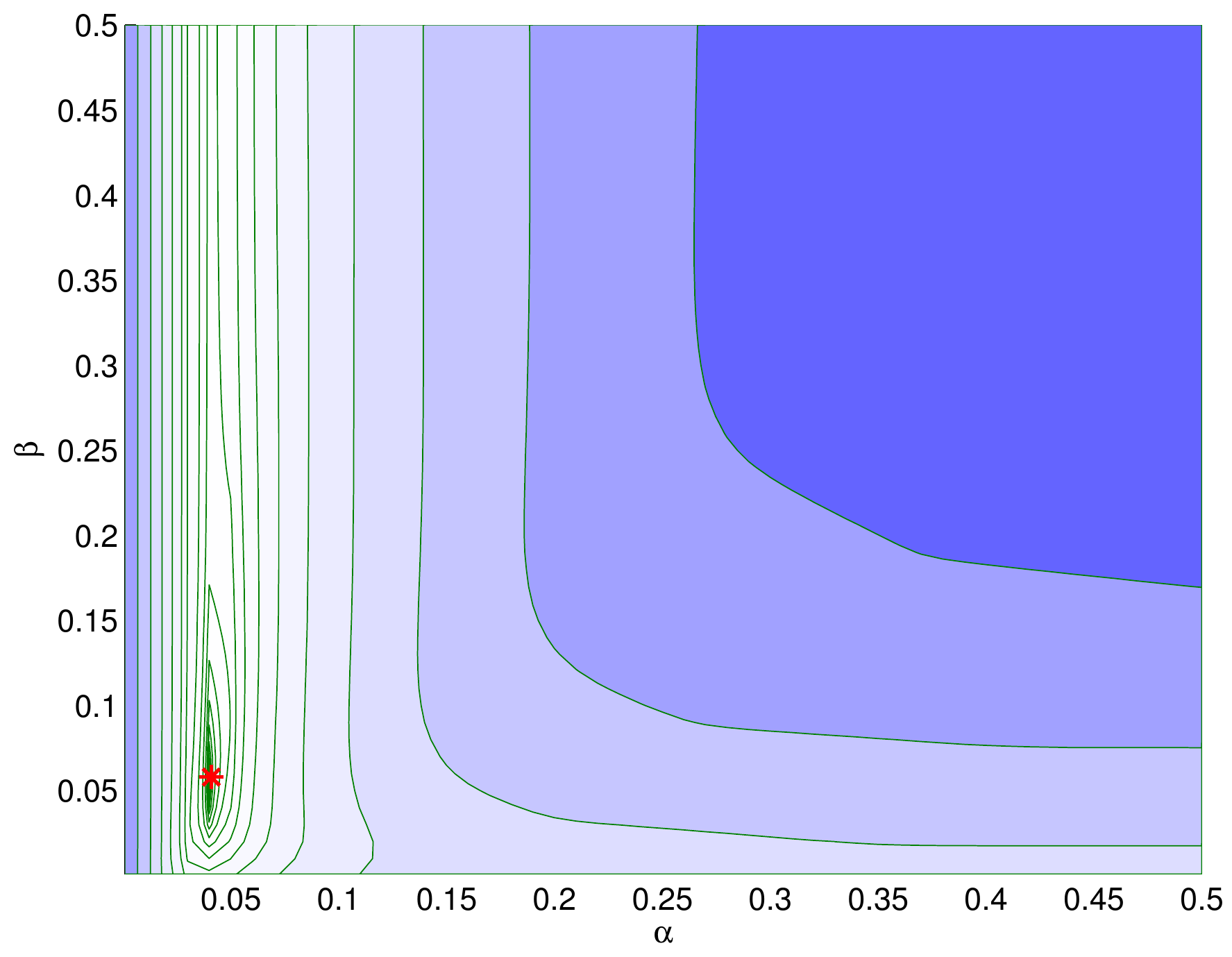}
        \caption{Parrot, $\TGV^2$, {\costltwo} cost functional}
    \end{subfigure}
    \hfill
    \begin{subfigure}[t]{0.47\textwidth}
        \includegraphics[width=\textwidth]{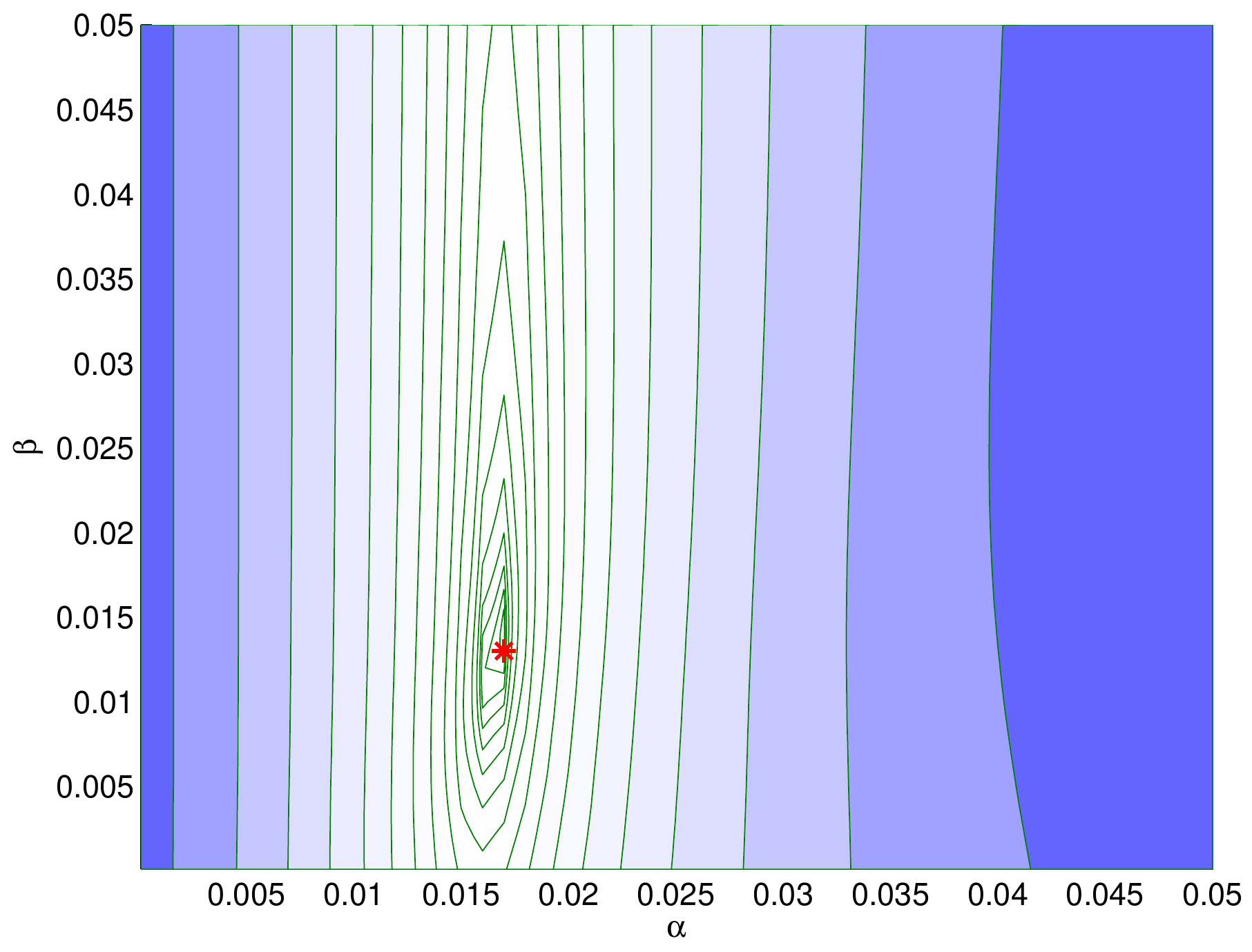}
        \caption{Uplands, $\TGV^2$, {\costltwo} cost functional}
    \end{subfigure}
    \caption{Cost functional value versus $\alpha$ for $\TGV^2$ denoising, 
        for both the parrot and uplands test images, 
        for both {\costltwo} and {\costhubertv}
        cost functionals.
        The illustrations are contour plots of function value versus $\alphavec=(\beta, \alpha)$.
        For the parrot image, optimal $\alphavec$ from
        \cite{tuomov-tgvlearn}
        for the initialisation $0.1/\DIMdomain$, resp.~$(1/\DIMdomain^2, 0.1/\DIMdomain)$, is indicated by an asterisk.
        For the landscape image, optimal $\alphavec$ from
        \cite{tuomov-tgvlearn}
        for the initialisation $0.01/\DIMdomain$, resp.~$(0.1/\DIMdomain^2, 0.01/\DIMdomain)$, is indicated by an asterisk.
        }
    \label{fig:landscape-tgv2}
\end{figure}

\section*{Acknowledgements}

This project has been supported by King Abdullah University of Science and Technology (KAUST) Award No.~KUK-I1-007-43, EPSRC grants Nr.~EP/J009539/1 ``Sparse \& Higher-order Image Restoration'', and Nr.~EP/M00483X/1 ``Efficient computational tools for inverse imaging problems'', the Escuela Politécnica Nacional de Quito under award PIS 12-14 and the MATHAmSud project SOCDE ``Sparse Optimal Control of Differential Equations''. While in Quito, T.~Valkonen has moreover been supported by a Prometeo scholarship of the Senescyt (Ecuadorian Ministry of Science, Technology, Education, and Innovation).

\section*{A data statement for the EPSRC}

This is a theoretical mathematics paper, and any data used merely serves as a demonstration of mathematically proven results. Moreover, photographs that are for all intents and purposes statistically comparable to the ones used for the final experiments, can easily be produced with a digital camera, or downloaded from the internet. This will provide a better evaluation of the results than the use of exactly the same data as we used.

%
%
 \providecommand{\homesiteprefix}{http://iki.fi/tuomov/mathematics}
  \providecommand{\eprint}[1]{\href{http://arxiv.org/abs/#1}{arXiv:#1}}

\end{document}